\documentclass[a4paper,12pt]{amsart}
\usepackage{geometry}                		
\usepackage{graphicx}				
\usepackage{amssymb}

\usepackage{latexsym,exscale,enumerate,amsfonts,amssymb,mathtools}
\usepackage{amsmath,amsthm,amsfonts,amssymb,amscd,stmaryrd,textcomp}
\usepackage[normalem]{ulem}
\usepackage{young}
\usepackage{thmtools,xcolor}
\usepackage{easybmat}

\usepackage{tikz-cd}

\definecolor{colormy}{rgb}{0.8,0.05,0.05}
\definecolor{mycolor}{rgb}{0.25,0.99,0.25}

\usepackage[all]{xy}
\SelectTips{cm}{}

\usepackage{tikz}
\usetikzlibrary{decorations.markings}
\usetikzlibrary{decorations.pathreplacing}
\usetikzlibrary{arrows,shapes,positioning}
\tikzstyle directed=[postaction={decorate,decoration={markings,
    mark=at position #1 with {\arrow{>}}}}]
\tikzstyle rdirected=[postaction={decorate,decoration={markings,
    mark=at position #1 with {\arrow{<}}}}]

\usepackage{hyperref}

\newcommand{\Hom}{\mathrm{Hom}}
\newcommand{\End}{\mathrm{End}}
\newcommand{\Ext}{\mathrm{Ext}}

\newcommand{\soc}{\mathrm{soc}}
\newcommand{\Char}{\mathrm{char}}

\newcommand{\Cha}{\mathrm{ch}}
\newcommand{\mo}{\mathrm{mod}}

\newcommand{\T}{\mathcal{T}}
\newcommand{\St}{\mathcal{S}t}

\def\C{{\mathcal C}}
\def\Co{{\mathbb C}}
\def\N{{\mathbb{Z}_{\geq 0}}}
\def\Z{{\mathbb Z}}

\def\O{\mathcal{O}}
\def\P{\mathcal{P}}

\def\pr{\mathrm{pr}}
\def\Ind{\mathrm{Ind}}
\def\B{\mathcal {B}}
\theoremstyle{definition}
\newtheorem{thm}{Theorem}[section]
\newtheorem{cor}[thm]{Corollary}
\newtheorem{lem}[thm]{Lemma}
\newtheorem{prop}[thm]{Proposition}

\theoremstyle{definition}
\newtheorem{rem}[thm]{Remark}
\newtheorem{examplecounter}{Example}

\numberwithin{equation}{section}

\title{Tilting modules and cellular categories}

\author{Henning Haahr Andersen}

\email{h.haahr.andersen@gmail.com}

\date{}							

\begin{document}

\begin{abstract}
In this paper we study categories of tilting modules in ``quasi-hereditary-like" categories, extending joint work with Stroppel and Tubbenhauer. We show that such categories form strictly object-adapted cellular categories. Then we use this fact to specify a subset of our cellular basis elements, which generates all morphisms. We describe our methods in details for tilting modules of a reductive algebraic group $G$ over a field of characteristic bigger than $2$, and also of its subgroup schemes $G_rT$. As concrete examples we exemplify our constructions for $G = SL_2$, and for tilting modules in the ordinary BGG category $\O$.
\end{abstract}

\maketitle
\tableofcontents



\section{Introduction}
Let $G$ be a reductive algebraic group over a field $k$ of characteristic $p>0$. We shall use the same methods as those developed in \cite{AST} to explore the tilting modules for $G$.  Note that in \cite{AST} we focus on the quantum case. However, the results (we shall mainly need those in Section 3 of \cite{AST}) carry over verbatim to the modular case. In particular the analogue of \cite{AST}, Theorem 3.9 
proves that for each tilting module $Q$ for $G$ the endomorphism algebra $\End_G(Q)$ carries a natural cellular structure. The present paper explores this fact and proves further results about category $\T(G)$ of tilting modules for $G$, more precisely about the morphisms in $\T(G)$. 

Our first result is a generalization of \cite{AST}, Theorem 3.9: We prove that the additive category $\T(G)$ consisting of all tilting modules for $G$ is an object-adapted cellular category in the sense of \cite{W} and \cite{EL}. Using the terminology in \cite{BS} this means that $\T(G)$ is the module category for an upper finite based quasi-hereditary algebra.  A very simple but nevertheless useful consequence of the cellularity of $\T(G)$ enable us then to single out a collection of special cellular basis elements in the $\Hom$-spaces between indecomposable tilting modules which generates all morphisms in $\T(G)$.

Secondly, we extend these results to include the case where $G$ is replaced by the subgroup scheme $G_rT$. Here $G_r$ denotes the $r$-th Frobenius kernel in $G$ and $T$ is a maximal torus. A module is tilting for $G_rT$ if and only if it is projective, so in this case we get a similar cellular structure on the category $\P_r$ of projective $G_rT$-modules. The periodicity in this category allows us in this case to exhibit a finite set of special cellular basis elements, which generates all morphisms in $\P_r$. 

It is well-known that the projective modules for $G_r$ (we may even take $r = 1$) contain the secrets of how to find the simple characters for $G$, see e.g. \cite{AJS}, \cite{RW}, and \cite{S}. This fact together with the close connection (at least for $p$ not too small, see Remark 4.2) between the categories $\T(G)$ and $\P_r$ made us interested in studying the properties of the morphisms in these two categories.

Our approach to these results works quite generally for ``quasi-hereditary-like" categories, for instance to tilting modules in the BGG category $\O = \O(\mathfrak g)$ for a complex semisimple Lie algebra $\mathfrak g$, and to tilting modules for the quantum group $U_q(\mathfrak g)$ at a root of unity $q$. We have (except for a short section dealing with category $\O$) chosen to formulate and prove our results in the modular case. The transition to tilting modules for a quantum group at a root of unity is straightforward.  

Once we have a set of generators for the morphisms in $\T(G)$ and $\P_r$ we want to find the relations among them. As we shall point out some relations are rather obvious. However, to get the full set of such relations is a much harder problem. We illustrate the complexity of this problem by giving some examples: We treat first the tilting modules in category $\O$ for $\mathfrak{sl}_3$,  and then in some details the projective $G_rT$-modules (especially for $r = 1, 2$) when $G = SL_2$. In both these cases the indecomposable tilting modules are multiplicity free. This simplifies considerably the problem but as we shall see even in these categories to get the relations among our generating modules requires rather detailed analysis of the structures of indecomposable tilting modules.

The group $SL_2$ is a natural starting point as well as a good test case, cf.  \cite{AT} which describes the category of tilting modules for the quantum group for $sl_2$ at a complex root of unity as the module category for a quotient of the zigzag quiver algebra. The recent preprint \cite{TW} which contains a deep analysis of the tilting category for $SL_2$ going much further than we do here (exploring for instance natural transformations between modular Jones--Wenzl projectors) was one of the motivating factors for me to undertake the present work. 

The paper is organized as follows. In Section 2 we deal with the category $\T(G)$ of tilting modules for a reductive algebraic group $G$. We introduce basic notations and collect the various results on representations of $G$ that we need to establish the cellularity of $\T(G)$. As already pointed out all ingredients needed for this can be found in \cite{AST}. But since we did not actually state the result in the generality we need here (in \cite{AST} we focused on proving  that endomorphism rings of  tilting modules are cellular algebras), we have carefully given the cellular data for $\T(G)$ and proved that the list of required properties are all satisfied. We also check that with a little more care  in our choice of bases we can obtain that  $\T(G)$ is a strictly object-adapted cellular category, a concept introduced in \cite{EL}. Finally, we exhibit among our cellular basis elements a set of special ones and prove that they generate all the morphisms in $\T(G)$.

Section 3 takes us on a detour to characteristic zero: We show that our procedure applies to the category of tilting modules in the BGG category $\O$. In particular this gives a set of generators for the morphisms in the principal block $\T(0)$ of this category. We take the opportunity to explain how we can significantly refine this set of generators in the case where all tilting modules in  $\T(0)$ are multiplicity free. As an illustrative example we treat the case $\mathfrak g = \mathfrak sl_3$. In this example we furthermore find the relations satisfied by our (refined) set of generators.

In Sections 4 and 5 we return to the modular case but now we fix an $r \geq 1$ and consider the subgroup scheme $G_rT$ of $G$. The tilting category for $G_rT$ coincides with the category $\P_r$ of projective $G_rT$-modules. We check that this category is cellular. It is also a strictly object-adapted cellular category except that the poset involved do not satisfy dcc (as is required in the definition given in \cite{EL}). This means that it is the module catogory for an essentially finite based quasi-hereditary algebra, cf. \cite{BS}.  Again we get a set of generators for the morphisms in $\P_r$, and in this case, up to tensoring with $1$-dimensional $G_rT$-modules, our set of generators is finite. 

The final section is devoted to the case $G = SL_2$. Here the principal block in the category $\P_1$ is a quotient of the zigzag quiver algebra with nodes parametrized by $\Z$. The relations are the same as the ones describing the tilting modules for the quantum group $U_q(\mathfrak sl_2)$ with parameter equal to a complex root of unity, cf. \cite{AT}. Now the complexity of $\P_r$ grows rapidly with $r$. However, for $G = SL_2$ the indecomposable  tilting modules for $G_rT$ are multiplicity free for all $r$. This allows us to describe a procedure for giving a nice set of generators for the morphisms in $\P_r$. We treat the case $\P_2$ in some details and in this case we also work out the relations among the generators.

\bigskip

\noindent \textbf{Acknowledgements:} I would like to thank D. Tubbenhauer for several email discussions and for making me aware of the paper \cite{EL}. I'm also grateful to the referee for his/her many useful comments and suggestions.

\section{Reductive  algebraic groups and their finite dimensional modules} \label{G}

We start by introducing some basic notations. The readers can find details in \cite{RAG}, Chapters II.1-3, II.11, and II.E (we point out the few instances where we deviate for the notation used there). 

\subsection{Basic Notation}Throughout this paper $G$ will denote a connected reductive algebraic group over an algebraically closed field $k$. We assume that $p = \Char \, k$ is positive.

We choose a maximal torus $T$ in $G$ and denote by $X =X(T)$ its character group. In the root system $R \subset X$ for $(G,T)$ we fix a set of positive roots $R^+$ and denote by $X^+ \subset X$  the corresponding cone of dominant characters. Then $R^+$ defines an ordering $\leq$ on $X$ and $X^+$. It also determines uniquely a Borel subgroup $B$ whose roots are the set of negative roots $-R^+$. The opposite Borel subgroup of $B$ will be denoted $B^+$. We have $B = UT$ and $B^+ = U^+T$, where $U$, respectively $U^+$, is the unipotent radical of $B$, respectively $B^+$.

Denote by $S$ the set of simple roots in $R^+$. If $\alpha$ is a root we set $\alpha^\vee$ equal to the corresponding coroot. Then the dominant cone $X^+$ is given by 
$$ X^+ = \{\lambda \in X | \langle \lambda, \alpha^\vee \rangle \geq 0 \text { for all } \alpha \in S\}.$$
We set $\rho = {\frac{1} {2}}  \sum_{\alpha \in R^+} \alpha$ and assume that $\rho \in X$. The fundamental alcove in $X^+$ is
$$ A =  \{\lambda \in X| 0 < \langle \lambda + \rho, \alpha^{\vee} \rangle < p \text { for all } \alpha \in R^+\}.$$
The closure $\bar A$ is the corresponding set with the inequalities replaced by $\leq$.

We set $\C(G)$ equal to the category of finite dimensional $G$-modules. 
If $K$ is a closed subgroup (or subgroup scheme) of $G$ we write similarly $\C(K)$ for the category of finite dimensional $K$-modules. The elements of $X$ identifies with the set of $1$-dimensional modules in $\C(T)$. These modules all extend uniquely to $B$ as well as to $B^+$. If $\lambda \in X$ we abuse notation and write also $\lambda $ for the $1$-dimensional $T$, $B$, or $B^+$-module determined by $\lambda$ . 

Recall that the category $\C(T)$ is semisimple, i.e. each $M \in \C(T)$ splits into a direct sum of $\lambda$'s. We set
$$ M_\lambda =\{m \in M | tm = \lambda(t) m, t \in X\},$$
and call this the $\lambda$ weight space of $M$. We say that $\lambda$ is a weight of $
M$ if $M_\lambda \neq 0$. Then we have 
$$ M = \oplus_\lambda M_\lambda,$$
with the sum running over all weights of $M$. We denote by $\Z[X]$ the group algebra of the additive group $X$ and define 
the character of M by
$$ \Cha \, M = \sum_\lambda \dim M_\lambda \, e^\lambda \in \Z[X].$$
Here the sum again runs over all weights of $M$.

\subsection{Standard, costandard and simple modules in $\C(G)$}
Let $H \leq K$ be closed subgroup schemes in $G$ and let $\Ind_H^K : \C(H) \rightarrow \C(K)$ denote the corresponding induction functor. In general this functor does not preserve finite dimensionality, but we shall only consider cases where it does. This in particular includes the case where $H = B$ and 
$K = G$, which we now consider. 

Let $\lambda \in X$. Then we set 
$$ \nabla(\lambda) = \Ind_B^G \lambda.$$
Recall that $\nabla(\lambda) \neq 0$ if and only if $\lambda \in X^+$. When $\lambda \in X^+$ the socle of $\nabla(\lambda)$ is simple and we shall write 
$$ L(\lambda) = \soc_G \nabla(\lambda).$$
Then the set $(L(\lambda))_{\lambda \in X^+}$ is a complete set of pairwise non-isomorphic simple objects in $\C(G)$. Moreover, by construction we have for $\lambda \in X^+$
$$ L(\lambda)_\lambda = \nabla(\lambda)_\lambda = k,$$
and all weights $ \mu$ of $\nabla(\lambda)$ (and hence also of $L(\lambda)$) satisfy $\mu \leq \lambda$. 

If $M \in \C(G)$ we denote by $[M:L(\lambda)]$ the composition factor multiplicity of $L(\lambda)$ in $M$. Then the above shows that
$$ [\nabla(\lambda): L(\lambda)] = 1 \text { and if } [\nabla(\lambda):L(\mu)] \neq 0 \text { then } \mu \leq \lambda.$$

Recall from \cite{RAG} II.1.16 that
 there is antiautomorphism $\tau$ of $G$, which is an involution, restricts to the identity on $T$ and takes each root subgroup $U_\alpha$  into $U_{-\alpha}, \; \alpha \in R$.  Then we define $^\tau M$ to be the linear dual $M^*$ of $M$ with $G$-action given by
$$ gh: m \mapsto h(\tau(g)m) \text { for all } g \in G, h \in M^*, m \in M.$$
The duality functor $M \mapsto \, ^\tau M$ on $\C(G)$ preserves weights and dimensions of weight spaces (i.e. $\Cha M = \Cha  \,^\tau M$). Therefore we see that $^\tau L(\lambda) \simeq L(\lambda)$ for all $\lambda \in X^+$. 

Define now
$$ \Delta(\lambda) =\, ^\tau\nabla(\lambda).$$
This is often called the Weyl module (with highest weight $\lambda$) because its character is given by $\Cha \Delta(\lambda) = \chi(\lambda)$, where $\chi(\lambda)$ is the Weyl character, i.e. the character of $\Delta (\lambda)$ is independent of $p$ and the same as in characteristic $0$.

A Weyl module is usually not simple (in contrast with what happens in characteristic zero). However, for some special weights the corresponding Weyl module is selfdual and simple. This includes the following cases:
\begin{equation}\label{simple Weyl}
 \text { If }  \lambda \in \bar A \cap X^+ \text { or if } \lambda = (p^r-1)\rho   \text { for some $r \geq 1$  then } \nabla(\lambda) = \Delta(\lambda) = L(\lambda).
\end{equation}
This follows from the strong linkage principle, \cite{SLP}, Theorem 1.

We set $St_r = L((p^r-1)\rho)$ and call this module the $r$-th Steinberg module. Note that by (\ref{simple Weyl}) we have also $St_r = \Delta ((p^r-1)\rho) = \nabla ((p^r-1)\rho).$

\subsection{ Tilting modules for $G$}
A crucial relation between Weyl and dual Weyl modules is:
\begin{equation} \label{ext-vanishing} 
\text { Let } \lambda, \mu \in X^+. \text { Then } \Ext_G^{i}(\Delta(\lambda), \nabla(\mu)) = \begin{cases} {k \text { if } \mu = \lambda \text { and } i=0,}\\ 0 \text { otherwise.}\end{cases} 
\end{equation}

A module $M \in \C(G)$ is said to have a $\nabla$-, respectively a $\Delta$-filtration, if there exists a sequence of submodules
$$ 0 = M_0 \subset M_1 \subset \cdots \subset M_r = M$$
with $M_i/M_{i-1} \simeq  \nabla (\lambda_i)$, respectively $\Delta(\mu_i)$, for some $\lambda_i, \mu_i \in X^+$, $i= 1,2, \cdots, r$.

If $M$ has a $\nabla$-filtration as above we set 
$$ (M:\nabla(\lambda)) = |\{i | \lambda_i = \lambda\}|,$$
and in case of a $\Delta$-filtration we define $(M:\Delta(\mu))$ analogously. It follows from (\ref{ext-vanishing})
 that we have:
\begin{equation} \label{nabla-mult}
\text { If $M$ has a $\nabla$-filtration then } (M:\nabla(\lambda))= \dim_k \Hom_G(\Delta(\lambda), M) \text { for all } \lambda \in X^+.
\end{equation} 
Likewise:
\begin{equation} \label{delta-mult}
\text { If $M$ has a $\Delta$-filtration then } (M:\Delta(\mu)) = \dim_k \Hom_G(M, \nabla(\mu)) \text { for all } \mu \in X^+.
\end{equation}

We say that $Q \in \C(G)$ is tilting if $Q$ has both a $\nabla$- and a $\Delta$-filtration. It is easy to check that if $Q$ is tilting then 
$$ (Q:\nabla(\lambda)) = (Q:\Delta(\lambda)) \text { for all } \lambda \in X^+.$$

Recall that we have the following classification of indecomposable tilting modules in $\C(G)$: For each $\lambda \in X^+$ there is a unique (up to isomorhisms) indecomposable tilting module $T(\lambda)$ with $T(\lambda)_\lambda = k $ and $T(\lambda)_\mu = 0$ unless $\mu \leq \lambda$, and this accounts for all indecomposable tilting modules. 

Hence if $Q \in \C(G)$ is tilting we have unique non-negative integers $(Q:T(\lambda))$ such that
$$Q = \bigoplus_{\lambda \in X^+} T(\lambda)^{\oplus (Q:T(\lambda))}.$$
Of course, $(Q:T(\lambda)) = 0$ for all but finitely many $\lambda \in X^+$. 

\begin{examplecounter} Note that (\ref{simple Weyl}) implies 
\begin{enumerate}
\item  $T(\lambda) = L(\lambda)$ for all $\lambda \in \bar A \cap X^+$.
\item  $T((p^r-1)\rho) = St_r$ for all $r \in \N$.
\end{enumerate}
\end{examplecounter}
For later use we record the following easy consequence of (\ref{ext-vanishing}), see \cite{A98}, Proposition 2.5 (or rather its proof):
\begin{prop}\label{split}
Let $Q \in \T(G)$ and $\lambda \in X^+$. Suppose $\lambda$ is a maximal weight of $Q$ and $\phi: T(\lambda) \rightarrow Q$ is a homomorphism, which is non-zero on $T(\lambda)_\lambda$. Then $\phi$ is a split injection.
\end{prop}

\subsection{The category of tilting modules for $G$}\label{tiltcat}

Let $\T(G)$ be the full subcategory of $\C(G)$ consisting of all tilting modules for $G$. Clearly, $\T(G)$ is an additive (but not abelian) category. It is a non-trivial fact, see  \cite{Ma}, Theorem 2, or alternatively  \cite{RAG}, Proposition II.E.7, that the tensor product of two tilting modules is again tilting. In this way $\T(G)$ thus becomes a monoidal category. 

It is clear that the dual of a tilting module is also tilting, i.e. the duality functor $M \mapsto \, ^\tau M$ on $\C(G)$ restricts to an endofunctor on $\T(G)$. Moreover, since this functor preserves characters we see that $^\tau T(\lambda) \simeq T(\lambda)$ for all $\lambda \in X^+$.  

We now choose for each $\lambda \in X^+$ an isomorphism $\phi_\lambda : T(\lambda) \rightarrow \, ^\tau T(\lambda)$. Let $Q \in \T(G)$. As $Q$ is a direct sum of $T(\lambda)$'s these isomorphims give us an isomorphism $\phi_Q : Q \rightarrow \, ^\tau Q$. We say that $\phi_Q$ is symmetric if it satisfies $\phi_Q(v)(v') = \phi_Q(v')(v)$ for all $v, v' \in Q$. 

Define then a contravariant functor $\dagger : \T(G) \rightarrow \T(G)$ by letting it be the identity on objects and given by the following recipe on morphisms.
$$ f^{\dagger} = \phi_P^{-1} \circ f^* \circ \phi_Q \text { for any } f\in \Hom_G(Q,P), \; P, Q \in \T(G).$$
Here $f^*$ denotes the linear dual of $f$. Clearly, $f^* \in \Hom_G(^\tau Q,\, ^\tau P)$. It is clear that this functor is contravariant. Indeed, let $P, Q, R \in \T(G)$. The definition immediately gives that if $ f \in \Hom_G(P, Q)$ and $g \in \Hom_G(Q, R)$ then 
$$ (g \circ f)^\dagger = f^\dagger \circ g^\dagger.$$

\begin{prop}\label{dagger}
Suppose $\phi_\lambda$ is symmetric for every $\lambda \in X^+$. Then $\dagger$ is an involution.
\end{prop}

\begin{proof}
Let $P \in \T(G)$. We claim
\begin{equation}\label{anti}
 \phi_P = \phi_P^* \circ Ev_P.
\end{equation}
This claim says that  $\phi_P(v)(v') = \phi_P(v')(v)$ for all $v, v' \in P$, i.e. that  $\phi_P$ is symmetric. So it follows from our assumption. 

Let also $Q \in \T(G)$ and take $ f \in \Hom_G(P, Q)$. To see that $f^{\dagger \dagger}= f$ we have to check $\phi_Q^{-1} \circ \phi_Q^*\circ  f^{**}\circ (\phi_P^{-1})^* \circ \phi_P = f$. Using (\ref{anti}) this is equivalent to the commutativity of the diagram
\[
\begin{tikzcd}
P \arrow{r}{f} \arrow[swap]{d}{Ev_P} & Q \arrow{d}{Ev_Q}\\
P^{**} \arrow[swap]{r}{f^{**}} & Q 
\end{tikzcd}
\]
and this is obvious.
\end{proof}

\begin{rem}\label{sym}
If $p > 2$ it is easy to satisfy the assumption on symmetry  in Proposition \ref{dagger}.  In fact, if $\phi_\lambda : T(\lambda) \rightarrow \, ^\tau T(\lambda)$ is an arbitrary isomorphism we define $\tilde \phi_\lambda$ by $\tilde \phi_\lambda (t)(t') = \phi_\lambda (t)(t') +  \phi_\lambda (t')(t),\; t,t' \in T(\lambda)$. Then $\tilde \phi_\lambda$ is clearly a symmetric homomorphism. It follows from Proposition \ref{split} that it is an injection (and hence an isomorphism since $\dim T(\lambda) = \dim \, ^\tau T(\lambda))$ if  it is non-zero on $T(\lambda)_\lambda$.  However, if $v \neq 0$ belongs to this $1$-dimensional weight space, then $\tilde \phi (v)(v) = 2  \phi (v)(v) \neq 0$.
\end{rem}

\subsection{Cellular categories} \label{cellular categories}

By the modular analogue of \cite{AST}, Theorem 3.9  we know that  if $Q \in \T(G)$ then the endomorphism algebra    $\End_G(Q)$ has a natural structure making it into a cellular algebra (in the sense of \cite{GL}, Definition 1.1). We shall now prove more generally that $\T(G)$ is a cellular category (see \cite{W}, Definition 2.1  for the definition of a cellular category or alternatively read it off from the proof of the following theorem). 
As will become clear all the necessary work for proving this result was already done in \cite{AST}, Section 3.

\begin{prop} \label{cellularity for G}
Suppose $p > 2$. Then $\T(G)$ is a cellular category.
\end{prop}

In the proof of this proposition we shall need the following families of morphisms in $\C(G)$. Let $\lambda \in X^+$ and suppose  $\varphi: M \to N$ is a morphism in $\C(G)$. Then $\varphi$ is in particular a $T$-homomorphism, i.e. it maps $M_\mu$ into $N_\mu$ for all $\mu \in X$. Let us denote by $\varphi_\mu: M_\mu \to N_\mu$ the restriction of $\varphi$ to $M_\mu$. We then say that $\varphi \in \C^{< \lambda}(G)$ if $\varphi_\mu = 0$ unless $\mu < \lambda$. 

We define $\C^{\leq \lambda}(G)$ similarly. Of course, we also have the analogous families $\T^{< \lambda}(G)$ and  $\T^{\leq \lambda}(G)$  of morphisms in the category $\T(G)$.

\begin{proof} In Subsection \ref{tiltcat} we equipped  $\T(G)$  with an anti-involution $\dagger$. 

We shall now define what will be our cell datum in $\T(G)$. It consists firstly of the poset $\Lambda =(X^+, \leq)$, see Section 2.1. Secondly, we need for each $Q \in \T(G)$ and each $\lambda \in X^+$ to give a finite set $K(Q,\lambda)$. We set
$$ K(Q,\lambda) = \{1,2, \cdots ,(Q:\Delta(\lambda))\}.$$
Recall that $(Q:\Delta(\lambda)) = \dim_k \Hom_G(\Delta(\lambda), Q)$. We choose a basis $\{g_1^\lambda(Q), g_2^\lambda(Q), \cdots , g_r^\lambda(Q) \}$ for $\Hom_G(\Delta(\lambda), Q)$ (so $r = (Q:\Delta(\lambda))$), and for each $i$ we lift $g_i^\lambda(Q)$ to $\bar g_i^\lambda(Q) \in \Hom_G(T(\lambda), Q)$. This is possible, since by (\ref{ext-vanishing}) we have $\Ext^1_G(T(\lambda)/\Delta(\lambda), Q) = 0$. Then we set $f_i^\lambda(Q) =  g_i^\lambda(Q)^\dagger \in \Hom_G(Q, \nabla(\lambda))$ and $\bar f_i^\lambda(Q) =  \bar g_i^\lambda(Q)^\dagger$.

These choices allow us now to give the third ingredient of our datum for $\T(G)$, namely for each pair $P,Q \in \T(G)$ and each $\lambda$ we define the injection
$$ C^\lambda(P,Q): K(P,\lambda) \times K(Q,\lambda) \rightarrow \Hom_G(P,Q)$$
by the recipe $C^\lambda(P,Q)(i,j) = c_{ij}^\lambda(P,Q) := \bar g_i^\lambda(Q) \circ \bar f_j^\lambda(P)$.

We claim: The datum $(\Lambda, \{K(Q,\lambda) | Q \in \T(G), \lambda \in \Lambda\}, \{C^\lambda(P,Q) | P,Q \in \T(G), \lambda \in \Lambda \})$ satisfies the following three properties (named C-I, C-II and C-III in \cite{W}, Definition 2.1):

\begin{enumerate}
\item  The images of the maps  $C^\lambda(P,Q), \lambda \in \Lambda$ give a basis for $\Hom_G(P,Q)$ for all $P,Q \in \T(G)$. \\
This is the content of \cite{AST}, Theorem 3.1.
\item
 $ c_{ij}^\lambda(P,Q)^\dagger = c_{ji}^\lambda(Q,P)  $ for all $\lambda, i, j, P, Q$.\\
This is clear from the definitions together with the fact that $\dagger$ is an anti-involution.
\item Let again $P, Q, R \in \T(G), \lambda \in \Lambda, i \in K(P,\lambda), j \in K(Q,\lambda)$. Take $\varphi \in \Hom_G(Q, R)$. Then
$$ \varphi \circ c_{ij}^\lambda(P,Q) = \sum_{m \in K(R, \lambda)} r_\varphi(i,m) c_{m,j}^\lambda(P,R) \; {(\mo \; \T^{<\lambda}(G))}$$
for some scalars $r_\varphi(i,m) \in k$ which are independent of $j$.

To see this we express for each $i \in K(Q, \lambda)$ the composite $\varphi \circ g_i^\lambda(Q)$ in the basis $(g_m^\lambda (R))_{m\in K(R,\lambda)}$ for $\Hom_G(\Delta(\lambda), R)$, i.e. we find $r_\varphi(i,m) \in k$ such that 
$\varphi \circ g_i^\lambda(Q) = \sum_m r_\varphi(i,m) g_m^\lambda(R)$. Then we have $\varphi \circ \bar g_i^\lambda(Q) -  \sum_m r_\varphi(i,m) \bar g_m^\lambda(R) \in \T^{<\lambda}(G)$. This proves the desired equality.
\end{enumerate}
\end{proof}

\begin{rem}
\begin{enumerate}
\item The proof of this proposition works just as well in the quantum case considered in \cite{AST}. Hence the category  of tilting modules for the quantum group $U_q$, $q$ a root of unity in any field of characteristic different from $2$, is also a cellular category. In particular, the endomorphism algebras of tilting modules are cellular algebras as proved in \cite{AST}, Theorem 3.9. 
\item Clearly, the argument used in Remark \ref{sym} does not work in characteristic $2$ and at the moment we have no replacement. Therefore Proposition \ref{cellularity for G}, Theorem \ref{SOACC for G} below and the consequences we derive of these results  all need the assumption $p \neq 2$. The same goes for \cite{AST}, Theorem 3.9.
\end{enumerate} 
\end{rem}

If we are a little more careful with the choices made in the proof of Proposition \ref{cellularity for G} we can strengthen the result to get that $\T(G)$ is in fact a ´´strictly object-adapted cellular category" (or an SOACC for short). We refer to \cite{EL}, Definition 2.4 for the definition of SOACC. This result should also be compared to \cite{BS}, Definition 5.1 and Remark 5.5, where instead of SOACC the terminology 'module category for an upper finite based quasi-hereditary algebra (with duality)' is used.
\begin{thm} \label{SOACC for G}
Suppose $p > 2$. Then
$\T(G)$ is a strictly object-adapted cellular category.
\end{thm}

\begin{proof} We consider the poset $(X^+, \leq)$ as a subset of the objects in $\T(G)$ by identifying $\lambda \in X^+$ with $T(\lambda)$. So now $T(\lambda) \leq T(\mu)$ iff $\lambda \leq \mu$. Going through the points in \cite{EL}, Definition 2.4 we let  $\lambda \in X^+$ and $P \in \T$. Then we identify our set $K(P, \lambda)$ with the set $M(T(\lambda), X)$, respectively  $E(X, T(\lambda))$, occurring in that definition by associating to each $i \in K(P,\lambda)$ the lift $\bar g_i^\lambda(P) \in \Hom_G(T(\lambda), P)$, respectively its dual $\bar f_i^\lambda(P)\in \Hom_G(P, T(\lambda))$. The only thing we then need to do to make sure that our cellular datum in Theorem \ref{cellularity for G} satisfies the conditions in this definition is to observe that, since $\lambda$ is the highest weight in $T(\lambda)$ and occurs with multiplicity $1$, see Subsection 2.3, any basis vector $g^\lambda(T(\lambda)) \in \Hom_G(\Delta(\lambda), T(\lambda)) \simeq k$ is an inclusion. We can therefore choose as  lift of $g^\lambda(T(\lambda))$  the identity $\bar g^\lambda(T(\lambda)) = id \in \End_G(T(\lambda))$ for any $ \lambda \in X^+$. 
\end{proof} 

\begin{rem} \label{choice}
It will turn out that it is convenient for the results in the following to always work with the identity on $T(\lambda)$ as our lift of the inclusion of $\Delta(\lambda) \hookrightarrow T(\lambda)$ as in the above proof. So \it {in the rest of this paper we will fix this choice}. We will also \it{assume $p \neq 2$ from now on}.
\end{rem}

\subsection{Generators for tilting homomorphisms}

Set now 
$$ B(G) = \{\bar g_i^\lambda(T(\mu)) | \lambda, \mu \in X^+, \lambda \leq \mu, i = 1, 2, \cdots ,(T(\mu):\Delta(\lambda)) \}.$$
By our choice in Remark \ref{choice} we see that $\bar g_i^\lambda(T(\mu)) = c_{i1}^\lambda(T(\lambda), T(\mu))$ (note that $K(T(\lambda), \lambda) = \{1\}$ so that the only possible second index here is $1$), i.e. the elements of $B(G)$ are certain special elements of our cellular bases. 

With this notation we have.

\begin{thm} \label{generators for G} The tuple $(\{T(\lambda)\mid \lambda \in X^+\}, B(G))$ generates $\T(G)$ as an additive, $k$-linear category with duality. On the level of morphisms this means that every morphism in $\T(G)$ is obtained from $B(G)$ by  taking direct sums, duals and compositions.
\end{thm}

\begin{proof} As we recalled in Subsection 2.3 the modules in $\T(G)$ split into direct sums of $T(\lambda)$'s. Hence to prove the theorem we have only left to check that the cellular basis elements $c^\lambda_{ij}(T(\nu), T(\mu))$ from Proposition \ref{cellularity for G} all belong to the set of morphisms generated by $B(G)$ for all $\lambda, \nu, \mu \in X^+$. Furthermore, since (in our notation in Section \ref{cellular categories}) $c^\lambda_{ij}(T(\nu), T(\mu)) = \bar g_i^\lambda(T(\mu)) \circ  \bar f_j^\lambda (T(\nu))$ and $\bar f_j^\lambda(T(\nu))$ by definition is the dual of $\bar g_j^\lambda(T(\nu))$  we see that elements in $\Hom_G(T(\mu), T(\nu))$ are indeed linear combinations of composites of elements from $B(G)$ with duals of such.
\end{proof}

\section{Category $\O$}

 This section deals with the BGG category $\O$ for a semisimple Lie algebra $\mathfrak g$ over $\Co$. We refer to \cite{Hu} for a thorough treatment of this category. Actually, we shall only consider a small subcategory inside $\O$, namely the principal block, see \cite{Hu}, 1.13. We denote this block $\O(0)$.  It turns out that the arguments used in Section \ref{G} applies just as well to the tilting modules in $\O(0)$. Let us, however, point out right away that neither $\O$ nor $\O(0)$ are monoidal categories (if we 
 a module in $\O$ by a finite dimensional $\mathfrak g$-module we stay inside $\O$, but this is not the case if we tensor two arbitrary modules in $\O$. Moreover, the block $\O(0)$ is not stable under tensoring by finite dimensional modules). Nevertheless, the arguments from Section 2 still leads to a set of generators for the morphisms in the subcategory consisting of tilting modules in $\O(0)$. We first establish this in general, then demonstrates how we can refine this set of generators in the multiplicity free case. Finally we show that for $\mathfrak g = \mathfrak {sl}_3$, we can write down the relations among the (refined set of)  generators.

\subsection{The subcategory of tilting modules in $\O(0)$}

First we fix notation: We let $\mathfrak h$ be a Cartan subalgebra in the semisimple comlex Lie algebra $\mathfrak g$. Denote by $\Delta(\lambda)$ the Verma module (with respect to a Borel subalgebra $\mathfrak b$ containing $\mathfrak h$) with highest weight $\lambda \in \mathfrak h^*$, and by $T(\lambda)$ the corresponding indecomposable tilting module in $\O$, cf. \cite{Hu}, 11.2. We let  $\T(0)$ be the subcategory of tilting modules in $\O(0)$,  i.e. $\T(0)$ is the additive category in which the indecomposable objects are $\{T(w \cdot 0) | w \in W\}$. Here $W$ is the Weyl group for $\mathfrak g$. We shall denote the Bruhat order on $W$ (\cite{Hu}, 5.2) by $\leq$. Then for $y, w \in W$ we have $y \leq w$ if and only if $y\cdot 0 \geq w \cdot 0$, where the latter ordering is induced by the set of positive roots in the root system for $(\mathfrak g, \mathfrak h)$ determined by $\mathfrak b$.

Using notation similar to Section \ref{G} and writing $x$ instead of $x \cdot 0$ we define 
$$ B(0) = \{ \bar g_i^x(T(y)) | x, y \in W,  x \geq y, \; i=1, 2, \cdots , (T(y):\Delta(x))\}.$$

Then we get exactly as in Theorem \ref{generators for G}:
\begin{thm} Every morphism in the additive category $\T(0)$ is obtained from  $B(0)$ by taking direct sums, duals and compositions.
\end{thm}

\subsection{The multiplicity free case}
We  say that $\T(0)$ is multiplicity free if $(T(y):\Delta(x)) \leq 1$ for all $x, y \in W$. In this case  we can refine the set of generators as follows. We denote by $\bar g^x(y)$ a lift of the unique (up to scalar) element of $\Hom_\O(\Delta(x), T(y))$ and set 
$$B'(0) = \{\bar g^x(y) | y \leq x, \; y \text { and $x$ are neighbors in the Bruhat graph for } W\}.$$ 
Then we have. 
\begin{thm} \label{generators in O} Suppose $\T(0)$ is multiplicity free. Then every morphism in $\T(0)$ is obtained from the set $B'(0)$ by taking direct sums, duals, and compositions.
\end{thm}

\begin{proof}
It is well-known that whenever $x \geq y$ we have a unique (up to scalar) homomorphism in $\Hom_{\O}(\Delta(x), \Delta(y))$, that this is a composite of homomorphisms between Verma modules corresponding to any sequence $x= x_0 < x_1 < \cdots  < x_r = y$ in the Bruhat graph, and that these homomorphisms are all injections, see \cite{Hu}, 4.2 It follows that we we may take $\bar g^x(y)  = \bar g^{x_{r-1}}(x_r) \circ  \bar g^{x_{r-2}}(x_{r-1}) \circ \cdots \circ  \bar g^{x_0}(x_1) $. Hence we can refine the set $B(0)$ from Theorem \ref{generators in O} to the smaller set $B'(0)$.
\end{proof}

\subsection{$\mathfrak g = \mathfrak {sl}_3$}
In this subsection we consider $\mathfrak g = \mathfrak {sl}_3$. It is easy to check that in this case $\T(0)$ is multiplicity free, so that Theorem \ref{generators in O} applies. We shall use this result together with some brute force calculations to obtain a presentation of the tilting category in terms of a quiver algebra with relations. Our calculations will also demonstrate that this approach will not generalize to higher rank.

Let $\mathfrak g =\mathfrak sl_3$. Denote by $s$ and $t$ the simple reflections in $W$ so that $W = \{1, s, t, st, ts, w_0\}$ where $w_0 = sts = tst$. If $ w \in W$ we denote by $e_w \in \End_\O(T(w))$ the identity on $T(w)$. Moreover, we set (using the notation from Theorem \ref{generators in O})
$$ u_1 = \bar g^{w_0}(st),  u_2 = \bar g^{w_0}(ts),  u_3 = \bar g^{st}(s), u_4 =  \bar g^{st}(t),$$
$$  u_5 = \bar g^{ts}(s), u_6 = \bar g^{ts}(t),  u_7 = \bar g^{s}(1), u_8 = \bar g^{t}(1).$$
Then $\{u_1, u_2, \cdots , u_8\}$ is the set $B'(0)$ defined in Subsection 3.2 for $\mathfrak {sl}_3$. 

We can picture these elements as arrows in the Bruhat graph for the Weyl group for type $A_2$ as follows:

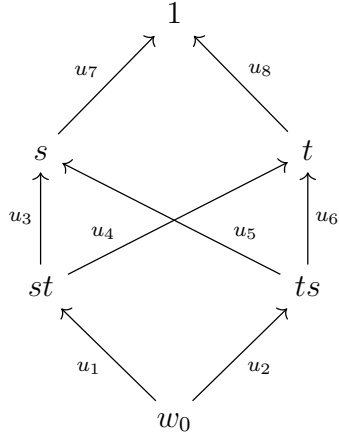
\[
\begin{tikzcd}[row sep=3.0em]
 & 1  \\
s \arrow{ur}{u_7} && t \arrow[swap]{ul}{u_8} \\ 
st \arrow{u}{u_3} \arrow[near start]{urr}{u_4 } && ts \arrow[swap, near start]{ull}{u_5} \arrow[swap]{u}{u_6} \\
& w_0 \arrow{ul}{u_1} \arrow[swap]{ur}{u_2} 
\end{tikzcd}
\]
\vskip .5cm
We let $d_i$ be the dual of $u_i$ for all $i$. In the above figure $d_i$ is the reverse arrow of $u_i$.
With this notation we have.

\begin{thm} \label{generators/relations}
The morphisms in the category $\T(0)$ for $\mathfrak sl_3$ are generated by the set
$$\{ e_w| w \in W\} \cup \{u_i | i=1, 2, \cdots ,8\} \cup  \{d_i | i=1, 2, \cdots ,8\}$$
subject to the following relations
\begin{enumerate}
\item $  u_3u_1 = u_4 u_2, \; u_4 u_1 = u_6u_2,\;  u_7u_3 = u_8u_4, \; u_8u_6 = u_7u_5,$
\item $ d_1d_3 = d_2 d_4, \; d_4 d_1 = d_2d_6,\;  d_3d_7 = d_4d_8, \; d_6d_8 = d_5d_7,$ 
\item $d_1u_1 =  d_2u_2 = d_4u_4 =  d_5u_5 = 0$,
\item $ d_3u_3 = a u_1d_1, \;  d_6u_6 = a u_2 d_2, \; d_6 u_4 = - a u_2 d_1, \; d_3u_5= - a u_1d_2, \; d_4 u_6 = - a u_1 d_2,\; d_5u_3= - a u_2d_1$ for some $a \in \Co^\times$,
\item$ d_7u_7 = b u_5d_5, \; d_8 u_8 = b u_4d_4, d_8 u_7 = b u_4d_3 + b u_6d_5 + r u_6u_2d_1d_3, \; d_7 u_8 = b u_3d_4 + b u_5d_6 + r u_3u_1d_2d_6$ for some $ b \in \Co^\times$,  $ r \in \Co$.

\end{enumerate}
\end{thm}

In the proof of this theorem we will use the translation functors in $\O$ defined in \cite{Hu}, 7.1. We shall in particular make use of their values on Verma modules and simple modules as described in \cite{Hu}, 7.6-7. We will also need wall-crossing functors. These are the composites of two translation functors, namely  the translation functor from our regular  block $\B(0)$ to the block determined by a semiregular weight on a wall of the dominant chamber, and the corresponding adjoint translation functor.

 We can determine all the indecomposable  tilting modules $T(w)$ by applying such wall-crossing functors: First we observe that $T(w_0) = \Delta(w_0) = L(w_0)$. Applying the wall crossing functor with respect to the $s$-wall to $T(w_0)$ we obtain $T(st)$ as a non-trivial extension of $\Delta(st)$ by $\Delta (w_0)$. Then applying the wall-crossing functor with respect to the $t$-wall to $T(st)$ gives us $T(s)$ with $\Delta$-factors $\Delta(s), \Delta (st), \Delta (ts)$, and $\Delta(w_0)$. Finally, applying the first wall crossing functor to $T(s)$ gives $T(1) \oplus T(st)$ showing that for each $w \in W$ the Verma module $\Delta(w)$ occurs exactly once as a $\Delta$-factor in $T(1)$. Interchanging $s$ and $t$  we obtain analogous
 statements about $T(ts)$ and $T(t)$. We observe that the socle  of $T(w)$ equals $L(w_0)$ for all $w$.
 
\begin{proof} We first observe that (2) is dual to (1). Likewise, the last two relations in (4) are duals of the previous $2$, and the last two relations in (5) are dual to each other. 

The relations (1) follow from the choices we made in the proof of Theorem \ref{generators in O}. 

The map $u_1$ takes $T(w_0) = L(w_0)$ to the socle of $T(st)$, and $d_1$ maps $T(st)$ onto its head $L(w_0) = T(w_0)$. Hence $d_1 u_1 = 0$. Similarly $d_2 u_2 = 0$, so that we have established the first two equalities in (3). 

To see that $d_4 u_4 = 0$ we observe that if $\theta$ denotes the wall-crossing functor with respect to the wall separating the $w_0$-chamber and the $st$-chamber, then $\theta$ takes the complex 
$$ L(w_0) = T(w_0) \hookrightarrow T(ts) \twoheadrightarrow L(w_0) = T(w_0),$$
(in which the maps come from the natural transformations $1 \to \theta$ and $\theta \to 1$ induced by adjointness of the two translation functors defining $\theta$), into the complex 
$$ T(st) \hookrightarrow T(t) \twoheadrightarrow T(st).$$
It is clear from the construction that the maps in this last complex are non-zero multiples of $u_4$ and $d_4$, respectively. Hence $d_4 u_4 = 0$. An analogous argument gives $d_5 u_5 = 0$ and we are done with (3). 

The properties of the cellular bases immediately show that the relations in (4) hold up to some scalars in $\Co$. We now check that these scalars are non-zero.  To see this choose a weight $\mu$ on the common wall of the $1$-chamber and the $s$-chamber. Set $\mu' = t \cdot \mu$ and $\mu'' = s \cdot \mu'$. Then the translation functor $T_0^\mu$ takes $u_3$ into an inclusion $T(\mu'') \oplus T(\mu'') \hookrightarrow T(\mu) \oplus T(\mu'')$. Likewise, $T_0^\mu$ takes $d_3$ into a surjection in the other direction. In other words, $T_0^\mu u_3$ identifies $T(\mu'') \oplus T(\mu'')$ with the socle of $T(\mu) \oplus T(\mu'')$ and $T_0^\mu d_3$ is mapping $T(\mu) \oplus T(\mu'')$ onto its head. We see that the composite is non-zero and hence so is $d_3 u_3$. The non-vanishing of the remaining scalars is checked by similar translations onto appropriate walls.  By symmetry (coming from swapping $s$ and $t$) the first two, respectively the middle two, scalars in (3) coincide. We shall see later that the two scalars involved sum to $0$. 

We now turn to the relations in (5). Here we choose a weight $\mu$ on the wall between the $t$-chamber and the $1$-chamber and set $\mu' = s \cdot \mu$ and $\mu'' = t \cdot \mu'$. Then we have an injection $T(\mu') \hookrightarrow T(\mu)$ and a dual surjection $T(\mu) \twoheadrightarrow T(\mu')$. The composite of these maps have image $T(\mu'') = L(\mu'')$. Applying the translation functor $T_\mu^0$  we obtain an injection $T(s) \hookrightarrow T(1)$ and a surjection $T(1) \twoheadrightarrow T(s)$. By our construction the injection is proportional to $u_7$ and the surjection to $d_7$. The image of the composite $d_7 u_7$ is therefore equal to $T_\mu^0$ applied to $T(\mu'')$, i.e. to $T(ts)$. This means that the diagram 
\[
\begin{tikzcd}
T(s) \arrow{r}{u_7} \arrow{d}{d_5} & T(1) \arrow{d}{d_7}\\
T(ts) \arrow{r}{u_5} & T(s) 
\end{tikzcd}
\]
commutes up to a non-zero scalar in $\Co$, i.e. that the first identity in (5) holds. A symmetrical argument gives the second relation in (5). Again by symmetry the scalars in the two first identities in (5) coincide.

We now prove the third relation in (5). First write $d_8 u_7$ as a linear combination of the $3$ cellular basis elements in $\Hom_\O(T(s), T(t))$
$$ d_8 u_7 = c u_4d_3 + c' u_6d_5 + c'' u_6u_2d_1d_3.$$
Applying $d_4$, respectively $d_6$, to this equation gives via the relations proved so far the identities
$$ c' a' = b a', \text { respectively } c a' + c' a = 0.$$
Here we have named the second scalar in (4) by $a'$ (because we haven't yet seen that it is $-a$).
Likewise, precomposing with $u_3$, respectively $u_5$, leads to
$$ ca + c'a' = 0, \text { respectively } c a' = b a'.$$
This implies (5) and shows also that $a' = -a$.

Now the only thing left to prove is that (1)-(5) are all relations among the given generators. To see this we need to show that any path (string of $u$'s and $d$'s, say from $x$ to $y$) may be written via the above relations as a linear combination of the known cellular basis elements in $\Hom_\O(T(x), T(y))$. In such a path we claim that the relations (1)-(5) allow us to move all the $d_j$ to the right of all $u_i$. In fact, suppose $d_j u_i$ occurs somewhere in our path. If it is one of the combinations in relation (3) then the path is zero. Otherwise, it occurs in (4) or (5) and these relations say that $d_ju_i$ is a linear combination of certain $u_{i'} d_{j'}$ or at worst also a term which is a product of two $u$'s and two $d$'s. So by repeated use of this argument we achieve the claim. This means that our path equals a linear combination of $\underline u \, \underline d$, where $\underline u$, respectively $\underline d$, is a path involving only $u$'s, respectively $d$'s. But by our choice of generators $\underline u \underline d$ is therefore one of the cellular basis elements of $\Hom_\O(T(x), T(y))$. 
\end{proof}

\begin{rem} \begin{enumerate}
\item Stroppel, \cite{Str}, 5.1.2, and Marko, \cite{FM}, Theorem 4.1, have obtained similar results on the endomorphism rings for projectives in $\O(0)$ for $\mathfrak {sl}_3(\Co)$.
\item The above proof reveals that the difficulty in the multiplicity free case in giving a presentation of the morphisms in $\T(0)$ by generators and relations is to find the relations, i.e. to determine the coefficients when writing elements of the form $d_j u_i$ as a linear combination of the elements from $B'(0)$. Cellularity of this basis limits which basis vectors occur with non-zero coefficients. However, even in the $\mathfrak {sl}_3$ case we did not determine the $3$ scalars $a, b, r$ (see relations (4) and (5) in Theorem \ref{generators/relations}). Moreover, our proof of Theorem \ref{generators/relations} relies on ad hoc calculations. Outside the multiplicity free case we have even bigger difficulties. So clearly the methods used in this section are insufficient to achieve such a presentation for $\T(0)$ in general. 
\end{enumerate}
\end{rem}

\section{Frobenius kernels and module categories for related subgroup schemes} \label{G_rT}

In this and the next section we shall prove results  for the subgroup schemes $G_rT$ in $G$ similar to those in Section \ref{G}. Recall that if $K$ is a subgroup scheme of $G$ then $\C(K)$ denotes the category of finite dimensional representations of $K$. 
We begin by recalling the definitions and main properties of $G_rT$ and its module category $\C(G_rT)$.  Again we refer to \cite{RAG} (mainly Chapters II. 9 and 11) for details.

\subsection{Subgroup schemes arising from Frobenius homomorphisms}
Let $F:G \to G$ denote the Frobenius endomorphism on $G$. Then for each $r \geq 1$ we denote by $G_r$ the kernel of $F^r$ considered as a closed subgroup scheme of $G$. Likewise we have corresponding subgroup schemes $T_r$, $B_r$, etc.
In addition, we shall also need the subgroup schemes
$$ G_rT = (F^r)^{-1}(T) \text { and } G_rB = (F^r)^{-1}(B).$$

If $M \in \C(G)$ we denote by $M^{(r)}$ the $r$-th Frobenius twist of $M$. This means that $M^{(r)} = M$ as $k$-vector space and that the action of $G$ on $M^{(r)}$ is given by $g m = F^r(g) m, g \in G, m \in M$. Note that $G_r$ acts trivially on $M^{(r)}$. Moreover, $\lambda$ is a weight of $M$ if and only if $p^r \lambda $  is a weight of  $M^{(r)}$. More precisely, $\Cha M^{(r)} = \sum_\lambda \dim M_\lambda e^{p^r \lambda}$, where the sum runs over the weights of $M$.

We define the Frobenius twists of modules in $\C(T)$, $\C(B)$, etc. similarly.

\subsection{Steinberg's tensor product theorem and simple modules in $\C(G_r)$} \label{G_r}

Let $r \geq 1$ and set 
$$X_r = \{\lambda \in X^+ | \langle \lambda, \alpha^{\vee}\rangle < p^r \text { for all } \alpha \in S\}.$$
The elements in $X_r$ are called the $r$-restricted weights. Note that $1$-restricted weights are also just called restricted weights.

Let $\lambda \in X$. 
We can then write a $p$-adic expansion $\lambda = \lambda^0 + p\lambda^1 + \cdots + p^r \lambda^r$ of $\lambda$ with $\lambda^i \in X_1$ for all $i$. In this notation Steinberg's tensor product theorem, \cite{St}, Theorem 1.1, says
$$ L(\lambda) = L(\lambda^0) \otimes L(\lambda^1)^{(1)} \otimes \cdots \otimes L(\lambda^r)^{(r)}.$$

The simple modules in $\C(G_r)$ are given by a theorem due to Curtis \cite{Cu}. It says 
\begin{equation} \label{Curtis} 
L(\lambda)_{|_{G_r}} \text { is simple for all $\lambda \in X_r$ 
and thus $X_r$ parametrizes the set of simple $G_r$-modules.}
\end{equation}

\subsection{$G_rT$-modules} \label{G_rT-mod}
Again we take $r \geq 1$. We write $\C_r$ short for $\C(G_rT)$, the category of finite dimensional $G_rT$-modules. If $\lambda \in X$ we write 
$$\lambda = \lambda^0 + p^r \lambda^1$$
 with $\lambda^0 \in X_r, \lambda^1 \in X$. Of course $\lambda^0$ and $\lambda^1$ depend heavily on $r$, but we have omitted $r$ from the notation.  It will always be clear from the context which $r$ we are working with.

We set  $L_r(\lambda) = L(\lambda^0)_{|_{G_rT}} \otimes p^r \lambda^1$. It follows from Section \ref{G_r} (see \cite{RAG} Proposition II.9.6 for details) that the simple modules in $\C_r$ are $(L_r(\lambda))_{\lambda \in X}$.

The standard and costandard objects in $\C_r$ may be defined similarly to the way we defined such objects in $\C(G)$ in Section \ref{G}. Namely, we set for each $\lambda \in X$
$$ \nabla_r(\lambda) = \Ind_{B_rT}^{G_rT} \lambda \text { and } \Delta_r(\lambda) = \Ind_{B_r^+T}^{G_rT} (\lambda - 2(p^r-1) \rho).$$
Alternatively, we have 
$$ \nabla_r(\lambda) = \Ind_{B}^{G_rB} \lambda \text { and } \Delta_r(\lambda) = \Ind_{B^+}^{G_rB^+} (\lambda - 2(p^r-1) \rho),$$
i.e. $\nabla_r(\lambda) $ and $\Delta_r(\lambda)$ are restrictions to $G_rT$ of corresponding $G_rB$-, respectively $G_rB^+$-modules.

The simple module $L_r(\lambda) \in \C_r$ is now realized as the socle of $\nabla_r(\lambda)$ or dually as the head of $\Delta_r(\lambda)$.

\begin{rem} Note that \cite{RAG}, Section II.9.1 uses a different notation for the standard and costandard modules. 
\end{rem}

As $T$-modules (in fact as $U_r^+T$-, respectively $U_rT$-modules) we have $\nabla_r(\lambda) \simeq k[U_r^+] \otimes \lambda$ and $\Delta_r(\lambda) \simeq k[U_r] \otimes (\lambda - 2(p^r-1)\rho)$. It follows that we have isomorphisms  $\nabla_r(\lambda) \simeq \nabla_r (\mu) \otimes (\lambda - \mu)$ in $\C(T)$ for all  $\lambda, \mu \in X$.  In the special case $\lambda = (p^r-1)\rho$ we have
isomorphisms of $G_rT$-modules 
\begin{equation} \label{St}
\nabla_r((p^r-1)\rho) \simeq St_r \simeq \Delta_r ((p^r-1)\rho).
\end{equation}
It follows that 
\begin{equation} \label{characters of baby-Vermas}
\Cha \nabla_r(\lambda) = \Cha \Delta_r(\lambda) = \chi ((p^r-1)\rho) e^{\lambda - (p^r-1)\rho}.
\end{equation}

Note that the antiautomorphism $\tau$ we used in Section 2 to define our duality in $C(G)$ restricts to an antiautomorphism on $G_rT$. Hence it gives us a duality on $\C_r$ as well. It is then a fact (which gives another explanation for the first equality in (\ref{characters of baby-Vermas})) that in $\C_r$ we have 
$$ \nabla_r(\lambda) \simeq \, ^\tau\Delta_r(\lambda)$$
for all $\lambda \in X$, see \cite{RAG}, II.9.3(5).

\subsection{Projective $G_rT$-modules}
A notable difference between $\C(G)$ and $\C_r$ is that the first contains no projective objects whereas the second has enough projectives. Moreover, in the category $\C_r$ to be projective is the same as to be injective, see \cite{RAG} Lemma.II.9.4. As we shall recall shortly, to be projective in $\C_r$ is also the same as being tilting.

Let $\lambda \in X$ and denote by $Q_r(\lambda)$ the projective cover  of $L_r(\lambda)$ in $\C_r$. Then $Q_r(\lambda)$ is also injective and since it is indecomposable it must be the injective envelope of some simple module in $\C_r$. As explained in \cite{RAG}, Section II.11.5 this simple module is $L_r(\lambda)$. 

By construction $\Delta_r(\lambda)$ is projective for $B_r$. In fact, $B_r = U_r T_r$ and as a $U_r$-module we have $\Delta_r(\lambda) \simeq k[U_r]$. Likewise $\nabla_r(\lambda)_{|_{U^+_r}} \simeq k[U_r^+]$ so that $\nabla_r(\lambda)$ is projective as a $B^+_r$-module. 

It follows from these observations that if $M \in \C_r$ has a $\Delta_r$-filtration then $M$ is projective for $B_r$, while if $M$ has a $\nabla_r$-filtration then $M$ is projective for $B_r^+$. It turns out that the converse is also true (\cite{RAG}, Proposition II.11.2), i.e. we have for $M \in \C_r$
\begin{equation} \label{filtration-criteria1}
M \text { is  $ B_r$-projective (equivalently injective) if and only if $M$  has a $ \Delta_r$-filtration}, 
\end{equation}
and
\begin{equation}\label{filtration-criteria2}
 M \text{ is $B_r^+$-projective (equivalently injective) if and only if $M$ has a $\nabla_r$-filtration}.
\end{equation}
In analogy with the definition of tilting modules in $\C(G)$ we say that $Q \in \C_r$ is tilting if  $Q$ has both a $\Delta_r$- and a $\nabla_r$-filtration. The multiplicities in such filtrations are denoted $(Q:\Delta_r(\lambda))$ and $(Q:\nabla_r(\lambda))$. 

By (\ref{filtration-criteria1}) and (\ref{filtration-criteria2}) we see that $Q$ is tilting if and only if $Q$ is projective for both $B_r$ and $B_r^+$. As $G_r = U_rB_r^+$ we conclude that
\begin{equation} \label{tilting=projective}
Q \in \C_r \text { is tilting if and only if $Q$ is projective if and only if $Q$ is injective. }
\end{equation}

Now of course we have 
\begin{equation}
Q \text { is projective  if and only if } Q = \oplus_{\lambda \in X} Q_r(\lambda)^{\oplus n_\lambda} \text { for some } n_\lambda \in \N \text { (almost all $0$)}.
\end{equation}

\subsection{The category $\P_r$ of projective $G_rT$-modules} \label{P_r}

We let $\P_r $ denote the subcategory of $\C_r$ consisting of all projective modules. By the results in the previous subsection this category is the same as the subcategory consisting of all tilting modules. We shall now deduce the exact relations between the PIM's in $\P_r$ and the indecomposable tilting modules.   

As in Section \ref{G_rT-mod} we write $ \lambda = \lambda^0 + p^r \lambda^1$
with $\lambda^0 \in X_r$ and $\lambda^1 \in X$. We now define
$$ \tilde \lambda = 2(p^r-1)\rho + w_0\lambda^0 + p^r \lambda^1,$$
where $w_0$ denotes the longest element in the Weyl group for $G$. Then the map $\lambda \mapsto \tilde \lambda$ is a bijection on $X$ which fixes all elements of $-\rho + p^rX$ and carries the box $p^r\mu +X_r$ onto the translated box  $p^r(\mu +\rho) - \rho + X_r$ for any $\mu \in X$.

Weight considerations (using e.g. (\ref{St})) show that the highest weight of $Q_r(\lambda)$ is $\tilde \lambda$.
This observation implies that the indecomposable tilting module $T_r(\tilde \lambda)$ in $\C_r$ with highest weight $\tilde \lambda$ is given by the following formula (see \cite{RAG}, II.11.3(1) for the second identity.
\begin{equation} \label{tilting=PIM}
 T_r(\tilde \lambda) = Q_r(\lambda) = Q_r(\lambda^0) \otimes p^r \lambda^1.
\end{equation}

\begin{examplecounter} \label{example2} Consider the special weight $(p^r-1)\rho \in X_r$.  For this weight we have
$$ T_r((p^r-1)\rho) = Q_r((p^r-1)\rho) = L_r((p^r-1)\rho) = \nabla_r((p^r-1)\rho)= \Delta_r((p^r-1)\rho).$$
In fact, these equations holds for all special weights, i.e. for all weights in $-\rho + p^rX$.
\end{examplecounter}

\begin{rem}
There are some interesting partially proved conjectures which connect indecomposable objects in $\P_r$ with objects in $\C(G)$ and $\T(G)$:
\begin{enumerate}
\item(The Humphreys-Verma conjecture, \cite{HV}, Theorem A)
Let $\lambda \in X_r$. There exists an object $\bar Q_r(\lambda) \in \C(G)$ such that $Q_r(\lambda) = \bar Q_r(\lambda)_{|_{G_rT}}$.

This conjecture is known to hold for $p \geq 2h-2$ in which case $\bar Q_r(\lambda)$ is the injective envelope of $L(\lambda)$ in a certain bounded subcategory of $\C(G)$, see \cite{RAG} II.11.11.

\item(Donkin's tilting conjecture, \cite{Do} Conjecture (2.2))
Let $\lambda \in X_r$. In the above notation $Q_r(\lambda) = T(\tilde \lambda)_{|_{G_rT}}$,  or equivalently $T_r(\tilde \lambda) =  T(\tilde \lambda)_{|_{G_rT}}$.

This conjecture was proved by Donkin (in \cite{Do}, Section 2) for $p \geq 2h-2$. It was recently shown to fail for $p=2$ for $G$ of type $G_2$, see \cite{BNPS}, Theorem 4.1.1.
\end{enumerate}
\end{rem}

\subsection{Reciprocity laws}

In analogy with (\ref{ext-vanishing}) we have 
\begin{equation} \label{ext-vanishing G_r}  \text {Let } \lambda, \mu \in X. \text { Then } \Ext^i_{G_rT}(\Delta_r(\lambda), \nabla_r(\mu)) = \begin{cases} {k \text { if } i = 0 \text { and } \mu = \lambda,}\\ {0 \text { otherwise.}}\end{cases} 
\end{equation}

We then also get analogues of (\ref{nabla-mult}) and (\ref{delta-mult}) in $\C_r$. This implies in particular the following reciprocity law. 
\begin{equation}
\text {Let $\lambda, \mu \in X$. Then } (Q_r(\lambda):\Delta_r(\mu)) = [\nabla_r(\mu):L_r(\lambda)].
\end{equation}
In fact, according to (\ref{ext-vanishing G_r}) the left-hand side equals $\dim \Hom_G(Q_r(\lambda), \nabla_r(\mu))$. This equals the right-hand side, because $Q_r(\lambda)$ is the projective cover of $L_r(\lambda)$.

Using (\ref{tilting=PIM}) we can also formulate this reciprocity in terms of indecomposable tilting modules: Let $\lambda, \mu \in X$. Then 
\begin{equation} \label{tilt-reciprocity}
(T_r(\tilde \lambda ):\Delta_r(\mu)) = [\nabla_r(\mu):L_r(\lambda)].
\end{equation}

\section{Homomorphisms in $\P_r$} \label{homomorphisms}
\subsection{Cellularity of $\P_r$} \label{cellularity for G_rT}
Arguing as in Section 2.4 we can now prove:
\begin{thm}\label{cellularity for P_r}
$\P_r$ is a cellular category.
\end{thm}

\begin{proof} Set $\Lambda = X$ and define for each $P \in \P_r$ and $\lambda \in \Lambda$ the set $K(P, \lambda) = \{1, 2, \cdots , (P:\Delta_r(\lambda))\}$. If also $Q \in \P_r$ define
 $$c_{ij}^\lambda(P,Q) = \bar g_i^\lambda(Q) \circ \bar f_j^\lambda(P),$$ 
where $\bar g_i^\lambda(Q)$, respectively $\bar f_j^\lambda(P)$, is a lift of a basis element $g_i^\lambda(Q) \in \Hom_{G_rT}(\Delta_r(\lambda), Q)$, respectively of a ``dual basis" element $f_j^\lambda(P) \in \Hom_{G_rT}(P, \nabla_r(\lambda))$ (obtained by applying the $\dagger$-functor analogous to the $G$-case from Subsection \ref{tiltcat}). This gives us exactly as in the proof of Theorem \ref{cellularity for G} a cellular datum for $\P_r$.
\end{proof}
.
\begin{rem}
In contrast to the cellular datum for $\T(G) $ presented in Theorem \ref{cellularity for G} the poset $(X, \leq)$  appearing in the cellular datum for $P_r$ in Theorem \ref{cellularity for P_r} is not lower (nor upper) finite. If we ignore this and remember our choice $\bar g^\lambda (T_r(\lambda)) = id_{T_r(\lambda)}$, see Remark \ref{choice},  then $\P_r$ is also an SOACC (cf. Theorem \ref{SOACC for G}), i.e. it satisfies all the other requirements in \cite{EL}, Definition 2.4. In the general framework treated in \cite{BS} it fits into the case of a module category for an 'essentially finite based quasi-hereditary algebra' in \cite{BS}, Subsection 5.2
\end{rem}

\subsection{Weight bounds} \label{weight bounds}
Let $P, Q \in \P_r$. Then, by (\ref{tilting=projective}) above,  $P$ and $Q$ are tilting modules and hence direct sums of certain $T_r(\lambda)$'s. The vector space $\Hom_{G_rT}(P,Q)$ is therefore a sum of certain $\Hom_{G_rT}(T_r(\lambda), T_r(\mu))$'s. To be precise 
$$ \Hom_{G_rT}(P,Q) \simeq \bigoplus_{\lambda, \mu \in X} \Hom_{G_rT}(T_r(\lambda), T_r(\mu))^{\oplus (P:T_r(\lambda))(Q:T_r(\mu))}.$$

We shall now prove that if $\lambda$ and $\mu$ are sufficiently far apart then  $\Hom_{G_rT}(T_r(\lambda), T_r(\mu)) = 0$. Recall the definition of $\tilde \lambda$ from Subsection \ref{P_r}.

\begin{lem} \label{bound}

Let $\lambda, \nu \in X$. If $(T_r(\tilde \lambda):\Delta_r(\nu)) \neq 0$ then $\lambda \leq \nu \leq \tilde \lambda$.
\end{lem}

\begin{proof}
Clearly, if $(T_r(\tilde \lambda):\Delta_r(\nu)) \neq 0$ then $\nu$ is a weight of $T_r(\tilde \lambda)$. This implies that $\nu \leq \tilde \lambda$. Moreover, by the reciprocity (\ref{tilt-reciprocity}) we have $(T_r(\tilde \lambda):\Delta_r(\nu)) = [\nabla_r(\nu):L_r(\lambda)]$. This gives the other inequality $\nu \geq \lambda$.
\end{proof}

\begin{rem}
The proof of Lemma \ref{bound} shows that $(T_r(\tilde \lambda):\Delta_r(\tilde \lambda)) = 1 = (T_r(\tilde \lambda):\Delta_r(\lambda))$. Consequently, the inequalities therein are as good as possible.
\end{rem}

\begin{prop} \label{hom-bounds}
Let $\lambda, \mu \in X$. If $\Hom_{G_rT}(T_r(\tilde \lambda), T_r(\tilde \mu)) \neq 0$ 
then $\mu \leq \tilde \lambda$ and $\lambda \leq \tilde \mu$.
\end{prop}

\begin{proof}
We have $\Hom_{G_rT}(T_r(\tilde \lambda), T_r(\tilde \mu)) \neq 0$ if and only if there exists $\nu \in X$ such that  $(T_r(\tilde \lambda):\Delta_r(\nu))$ and  $(T_r(\tilde \mu):\Delta_r(\nu))$ are both non-zero. Then by Lemma \ref{bound} the non-vanishing of $\Hom_{G_rT}(T_r(\tilde \lambda), T_r(\tilde \mu))$ implies that $ \mu \leq \nu \leq \tilde \lambda$ and $\lambda \leq \nu \leq \tilde \mu$.
\end{proof}

\begin{rem}
\begin{enumerate}
\item The bounds in Proposition \ref{hom-bounds} are always achieved: Suppose $\mu = \tilde \lambda$. Then $\Hom_{G_rT}(T_r(\tilde \lambda), T_r(\tilde \mu)) = \Hom_{G_rT}(T_r(\mu), T_r(\tilde \mu)) = \Hom_{G_rT}(T_r(\mu), Q_r(\mu)) = k$, 
because $[T_r(\mu):L_r(\mu)] = 1$. Likewise, if $ \lambda = \tilde \mu$ then   $\Hom_{G_rT}(T_r(\tilde \lambda), T_r(\tilde \mu)) = \Hom_{G_rT}(T_r(\tilde \lambda), T_r(\lambda)) = k$. 
\item Suppose $\lambda \in -\rho + p^r X$. Then $\lambda = \tilde  \lambda$ so that in this case the proposition says that 
$$\Hom_{G_rT}(T_r(\tilde \lambda), T_r(\tilde \mu))= \begin{cases} {k \text { if } \mu = \lambda,}\\ {0 \; \text { otherwise.}} \end{cases}$$

This can also be seen via Example \ref{example2}.
\item The proposition implies that if $\Hom_{G_rT}(T_r(\nu), T_r(\eta)) \neq 0$ then $ -2(p^r-1)\rho \leq  \nu - \eta \leq 2(p^r-1)\rho$. Note that there exists weights $\eta$ and $\mu$ for which these bounds are realized: It is easy to check that for instance $\Hom_{G_rT}(T_r(0), T_r(2(p^r-1)\rho)) = k$.
\item Combined with the fact that $T_r(\nu + p^r \eta) \simeq T_r(\nu) \otimes (p^r\eta)$ for all $\nu, \eta \in X$, Proposition \ref{hom-bounds} reduces the problem of finding all homomorphism spaces in $\P_r$ to a finite one: It is enough to determine $\Hom_{G_rT}(T_r(\tilde \lambda), T_r(\tilde \mu))$ for the finite set of pairs $(\lambda, \mu)$, where $\lambda \in X_r$ and $\mu$ satisfies $\mu \leq \tilde \lambda $ and $\lambda \leq \tilde \mu$.
\end{enumerate}
\end{rem}

\subsection{A set of generators}
The results in Section \ref{weight bounds} combined with the cellularity of $\P_r$ allow us now to single out a finite set of generators for the family of homomorphisms in $\P_r$. Using notation as in the proof of Theorem \ref{cellularity for P_r} we set 
$$ B_r ^\lambda (\mu) = \{\bar g_i^\lambda(T_r(\mu)) | i \in K(T_r(\mu), \lambda) \}$$
for all $\lambda, \mu \in X$. Note that by definition the lift $\bar f^\lambda(T_r(\lambda))$ of the basis element $f^\lambda(T_r(\lambda) \in \Hom_{G_rT}(T_r(\lambda), \nabla_r(\lambda)) = k$ is the dual of $\bar g^\lambda(T_r(\lambda))$, which we in accordance with Remark \ref{choice} choose to be the identity on $T_r(\lambda)$.  Thus $B_r^\lambda(\mu)$ consists of cellular basis elements.

Set now
$$  B_r = \bigcup_{\lambda \in X_r; \lambda \leq \mu \leq \tilde \lambda} B_r^\lambda(\mu).$$
Then we have the following analogue of Theorem \ref{generators for G}.
\begin{thm} \label{generators for G_r} The tuple $(\{T_r(\lambda) \mid \lambda \in X_r\}, B_r)$ generates $\P_r$ as an additive, $k$-linear category with duality. On the level of morphisms this means that every morphism in $\P_r$ is obtained from $B_r$ by taking direct sums, duals and compositions, and by tensoring with elements of $p^rX$. 
\end{thm}

\begin{proof} Just like in the proof of Theorem \ref{generators for G} we see that we are done if we  check that every cellular basis element $c^\lambda_{ij}(T_r(\nu), T_r(\mu))$ is generated by $B_r$ (in this case it means: obtained from $B_r$ by taking direct sums, duals and compositions, and by tensoring with elements of $p^rX$) for all $\lambda, \nu, \mu \in X$. Furthermore, since  $c^\lambda_{ij}(T_r(\nu), T_r(\mu)) = \bar g_i^\lambda(T_r(\mu)) \circ  \bar f_j^\lambda (T_r(\nu))$ we are done if we check that all $\bar g_i^\lambda(T_r(\mu))$  belong to the set generated by $B_r$. Note that here $\mu \geq \lambda$. By tensoring with an appropriate element of $p^rX$ we may assume $\lambda \in X_r$. Finally, Proposition \ref{hom-bounds} ensures that $K(T_r(\mu), \lambda)$ is empty unless $\mu \leq \tilde \lambda$.
\end{proof}

\subsection{The strong linkage principle in $ \C_r$ \and $\T_r$}\label{SLP}
If $\alpha \in R$ then the reflection $s_\alpha$ acts on $X$ by $s_\alpha \lambda = \lambda - \langle \lambda, \alpha^\vee \rangle \alpha, \lambda \in X$. For each $m \in \Z$ we have a corresponding affine reflection $s_{\alpha, m}$ given by $s_{\alpha, m} \lambda = s_\alpha \lambda + mp\alpha$. The affine Weyl group $W_p$ is the group generated by these affine reflections for $\alpha \in R, m \in \Z$. The dot action of $W_p$ on $X$ is then the above action shifted by $-\rho$, i.e. $w \cdot \lambda = w(\lambda + \rho) - \rho, w \in W_p, \lambda \in X$.  

Let $\mu, \lambda \in X$. We say that $\mu$ is strongly linked to $\lambda$  (this is the order relation denoted $\uparrow$ in \cite{RAG}, II.6.4)  if there exists a sequence $\mu = \mu_0 \leq \mu_1 \leq \cdots  \leq \mu_r = \lambda$ in $X$ with $\mu_{i+1}$ obtained (via the dot action) from $\mu_i$ by an affine reflection in $W_p$. Note that strongly linked weights are in particular linked in the sense that they belong to the same orbit of $W_p$ under the dot action.

The strong linkage principle in $\C_r$ may be formulated as follows, see \cite{RAG}, II.9.15.
\begin{equation} \label{linkage}
\text {Let } \lambda, \mu \in X. \text { If } [\Delta_r(\lambda): L_r(\mu)] \neq 0 \text { then $\mu$ is strongly linked to $\lambda$}.
\end{equation}
Here we can of course replace $\Delta_r$ by $\nabla_r$.

When we combine (\ref{linkage}) with (\ref{tilt-reciprocity}) we get (using notation as in Section \ref{P_r})
\begin{equation} \label{bounds on Delta-factors}
\text {Let  $\lambda, \nu \in X$. If $(T_r(\tilde \lambda):\Delta_r(\nu)) \neq 0$ then $\lambda$ is strongly linked to $\nu$.}
\end{equation}
We immediately get.
\begin{prop}
Let $\lambda, \mu \in X$. If $\Hom_{G_rT}(T_r(\tilde \lambda), T_r(\tilde \mu)) \neq 0$ then $\mu \in W_p \cdot \lambda$.
\end{prop}

\begin{rem}
Note that $\tilde \lambda = w_0 \cdot \lambda + p^r(\lambda^1 -w_0 \cdot \lambda^1)$ so that since $\lambda^1 - w_0 \cdot \lambda^1 \in \Z R$ we have $\tilde \lambda \in W_p \cdot \lambda$. Hence we can replace $\tilde \lambda$ and $\tilde \mu$ by $\lambda$ and $\mu$ in this proposition.
\end{rem}

\subsection{The translation principle in $\C_r$ and $\P_r$}
It follows from the results in the previous subsection that all composition factors of an indecomposable tilting module $T_r(\lambda)$ have highest weights in $W_p \cdot \lambda$. This is more generally true (e.g. because tilting modules are the same as projective modules) for the composition factors of any indecomposable module in $\C_r$. In other words, since $\bar A$ is a fundamental domain for the action by $W_p$ on $X$ we can split any module $M \in C_r$ as
$$ M = \bigoplus_{\lambda \in \bar A} \pr_\lambda(M).$$
Here $\pr_\lambda(M)$ is the largest submodule of $M$ such that its composition factors all have highest weights in $W_p \cdot \lambda$. This allows us for each $\lambda, \mu \in \bar A$ to define translation functors $T_\mu^\lambda$ on $\C_r$ and $\P_r$, see \cite{RAG}, II.9.22.
We set $\C_r^\lambda = \pr_\lambda(\C_r)$.  If $\lambda, \mu \in A$ then $T_\mu^\lambda$ is an equivalence between $C_r^\mu$ and $\C_r^\lambda$. A special consequence of this is:

\begin{prop} 
Let $\lambda, \mu \in A$. Then $\Hom_{G_rT}(T_r(x \cdot \lambda), T_r(y \cdot \lambda)) \simeq \Hom_{G_rT}(T_r(x \cdot \mu), T_r(y \cdot \mu))$ for all $x, y \in W_p$.
\end{prop}
Note that $A = \emptyset$ if $p$ is less than the Coxeter number for $R$. Hence this proposition is empty for small primes.

Another well known result (see \cite{RAG} II.11.10(1)) is.
\begin{prop}\label{translate st}
 Let $\lambda \in \bar A$ and $\mu \in X$. Then $T_{-\rho}^\lambda  T_r(-\rho + p^r\mu)\simeq  T_r(\lambda + p^r\mu)$.
\end{prop}

Note that by (\ref{St})  we have that $T_r(-\rho + p^r \mu) = \Delta_r(-\rho + p^r \mu)$. From the well known behavior of translation functors on standard modules in $C_r$ we get
\begin{cor}\label{tilting above st}  Let $\lambda \in \bar A$ and $\mu \in X$. Then 
$$(T_r(\lambda + p^r \mu) : \Delta_r(\nu)) = \begin{cases} {1 \text { if } \nu \in - \rho + p^r\mu + W\lambda,} \\ { 0 \text { otherwise. }} \end{cases}$$
\end{cor}

In fact, Proposition \ref{translate st} and Corollary \ref{tilting above st} are special cases of the following more general result.
\begin{prop}\label{translated tiltings}
Let $\lambda \in A$ and $\mu \in \bar A$. Suppose $w$ is an element of $W_p$ for which $w\cdot \lambda$ is maximal in the set $\{wx\cdot \lambda | x\cdot \mu = \mu\}$. Then
\begin{enumerate}
\item $T_\mu^\lambda T_r(w\cdot \mu) \simeq T_r(\lambda)$.
\item $(T_r(w\cdot \lambda):\Delta_r(y\cdot \lambda)) = (T_r(w\cdot \mu):\Delta_r(y\cdot \mu))$ for $y \in W_p$.
\end{enumerate}
\end{prop}

\begin{proof}
$(1)$ is the infinitesimal analogue of \cite{A00}, Proposition 5.2. Then we get (2) from the adjointness of $T_\mu^\lambda$ and $T_\lambda^\mu$ as follows:  $(T_r(w\cdot \lambda):\Delta_r(y\cdot \lambda)) = \dim_k \Hom_{G_rT}(\Delta_r(y\cdot \lambda), T_r(w\cdot \lambda)) = (T_r(w\cdot \mu):T_\lambda^\mu \Delta_r(y\cdot \lambda)) = (T_r(w\cdot \mu): \Delta_r(y\cdot \mu))$ because $T_\lambda^\mu \Delta_r(y\cdot \lambda) \simeq \Delta_r(y\cdot \mu)$.
\end{proof}

Using Proposition \ref{translated tiltings} we can now strengthen
 the conditions in (\ref{bounds on Delta-factors}).

 Let $\lambda \in X$ be $p$-regular, i.e. $\lambda \in W_p \cdot A$. If $A'$ is an arbitrary alcove, then we denote by $\lambda_{A'}$ the unique element in $A' \cap W_p \cdot \lambda$. Then we have. 
\begin{prop}
If $(T_r(\tilde \lambda):\Delta_r(\nu)) \neq 0$ for some $\nu \in X$ then $\lambda$ is strongly linked to $\nu$ and $\nu$ is strongly linked to $\tilde \lambda$.
\end{prop}

\begin{proof} We have left to check that if (for arbitrary $\lambda, \nu \in X$)  we have $(T_r(\lambda):\Delta_r(\nu)) \neq 0 $ then $\nu$ is strongly linked to $\lambda$. By tensoring with an appropriate element of $p^rX$ we may assume $\lambda \in X_r$. Denote by $A'$ the alcove containing $\lambda$. If $A' = A$ we are done by Corollary \ref{tilting above st}. So suppose $\lambda > \lambda_A$ and assume inductively that the statement holds for all weights of the form $\lambda_{A''}$ for which $A'' \subset X_r$ and $\lambda_{A''} < \lambda$. Choose now $A''$ to be an alcove in $X_r$ which shares a wall with $A'$ and satisfies $\lambda_{A''} < \lambda$. Let $\mu$ be a weight in the interior of the common wall of $A''$ and $A'$. By Proposition \ref{translated tiltings} we see that $T_r(\lambda)$ is a summand of $T_\mu^\lambda T_\lambda^\mu T_r(\lambda_{A''})$. Then $(T_r(\lambda):\Delta_r(\nu)) \leq (T_\mu^\lambda T_\lambda^\mu T_r(\lambda_{A''}):\Delta_r(\nu)) = (T_r(\lambda_{A''}):\Delta_r(\nu)) + (T_r(\lambda_{A''}): \Delta_r(\nu'))$ where $\nu' $ is the weight different from $\nu$ for which $\Delta_r(\nu'')$ is a $\Delta_r$-factor of $T_\mu^\lambda T_\lambda^\mu \Delta_r(\nu)$. The statement now follows from the induction hypothesis.
\end{proof}

Let $\P_r(0)$ denote the principal block in $\P_r$ corresponding to $0 \in A$, i.e. $\P_r(0)$ is the subcategory of $\P_r$ whose indecomposable modules are $(T_r(w \cdot 0))_{w \in W_p}$. Set $\Lambda(0) = W_p \cdot 0$ and let $\leq_{SL}$ denote the ordering on $\Lambda(0)$ given by $\lambda \leq_{SL} \mu$ iff $\lambda$ is strongly linked to $\mu$. Then we get from the above.
\begin{cor} \begin{enumerate}
\item $\P_r(0)$ is a cellular category with weight poset $(\Lambda(0), \leq_{SL})$.
 \item  $B_r(0) = \{\bar g_i^{x\cdot 0}(T_r(y\cdot 0)) | x, y \in W_p, x\cdot 0 \leq_{SL} y\cdot 0, \; i= 1, 2,
\cdots, (T_r(y\cdot 0):\Delta_r(x\cdot 0))\}$ generates the morphisms in $\P_r(0)$.
\end{enumerate}
\end{cor}

\subsection{The Steinberg linkage class in $\C_r$ and $\P_r$}
In this subsection we shall carry over to $G_rT$ some of the results in \cite{A18}.  Consider the special weight $(p-1)\rho \in X$. As $s_\alpha \cdot (p-1)\rho = (p-1)\rho - p\alpha$ for all simple roots $\alpha$  we see that the linkage component in  $\C_r$ corresponding to $(p-1)\rho$ consists of all $M \in \C_r$ whose composition factors have highest weights in $(p-1)\rho + p\Z R$. For the purposes in this paper it will be convenient to consider the (possibly bigger) subcategory consisting of all $M \in \C_r$ with composition factors belonging to $\{L_r(\lambda) | \lambda + \rho \in pX\}$. We shall denote this component of $\C_r$ by $\St_r$, and by $\P\St_r$ the subcategory of $\P_r$ consisting of those $M \in \P_r$ which belong to $\St_r$.

If $r = 1$ then we have:
$$\St_1 \text { is a semisimple category with simple modules } L_1(-\rho + p \lambda), \; \lambda \in X.$$
Note that in fact $L_1(-\rho + p\lambda)\simeq St_1 \otimes (p (\lambda - \rho)) \simeq \Delta_r(-\rho + p \lambda) = T_1(-\rho + p \lambda)$, so that $\St_1$ is contained in $\P_1$, i.e. $\P\St_1 = \St_1$.

Suppose $r > 1$. The Frobenius homomorphism on $G$ restricts to a homomorphism $G_r \to G_{r-1}$. Via this homomorphism we can make any $M \in \C_{r-1}$ into a $G_r$-module. The resulting module in $\C_r$ is denoted $M^{(1)}$. We then define a functor 
$$\Phi_{r-1} : \C_{r-1} \to \C_r \text { by } \; \Phi_{r-1}(M) = St_1 \otimes M^{(1)}. $$

We observe that $\Phi_{r-1}$ takes values in $\St_r$. In fact, it gives an equivalence of categories 
\begin{equation} \label{equiv-St} \Phi_{r-1} : \C_{r-1} \rightarrow \St_r
\end{equation}
with inverse functor $\Hom_{G_1}(St_1, -)$, cf. \cite{A18}, Theorem 3.1.
This equivalence of categories carries simple modules to simple modules, (co)standard modules to (co)standard modules, and tilting modules to tilting modules. In particular, we see that the restriction of $\Phi_{r-1}$ to $\P_{r-1}$ gives an equivalence between $\P_{r-1}$ and the subcategory $\P\St_r$ of $\P_r$. 

Let us also record the following consequences of the above.
\begin{prop} Let $r > 1$ and take $\lambda \in X$. Then the functor $\Phi_{r-1}$ carries the tilting module $T_{r-1}(\lambda) \in \P_{r-1}$ to $T_r((p-1)\rho + p \lambda) \in \P_r$, and if also $\mu \in X$ we get an isomorphism of $k$-vector spaces
$$ \Hom_{G_{r-1}T}(T_{r-1}(\lambda), T_{r-1}(\mu)) \simeq \Hom_{G_rT}(T_r((p-1)\rho +p\lambda), T_r((p-1)\rho +p\mu)).$$
\end{prop}

Finally, we observe that for $r>1$ we have a sequence of full subcategories of $\St_r$
$$ \St_r^r \subset \St_r^{r-1} \subset \cdots \subset \St_r^2 \subset \St_r$$
defined inductively by letting $\St_r^i$ denote the subcategory of $\C_r$,  which via  $\Phi_{r-1}$ is equivalent to the subcategory $\St_{r-1}^{i-1}$ of $\C_{r-1}$.

The smallest of these, $\St_r^r$, is a semisimple category with simple objects $L_r(-\rho + p^r\lambda)$, $\lambda \in X$, compare Remark 5.6(2).

\section{The $SL_2$-case}
In this section we take $G=SL_2$. Then $X = \Z$, $R = \{\pm 2\}$ and we choose $R^+ = \{2\}$ so that $X^+ = \N$. We shall describe the corresponding categories $\P_r$, in particular the homomorphisms in these,  by using the cellular structures discussed in the previous sections. 

Note that in this case the affine Weyl group $W_p$ is the infinite dihedral group with generators $s$ and $t$, the reflections (with respect to the dot actions) in $-1$ and $p-1$, respectively. The alcove $A$ is the interval $[0, p-2]$.

\subsection{r=1} \label{r=1}
There are only two different components of $\P_1$, namely the one corresponding to the orbit $W_p \cdot 0$ and the one corresponding to  $W_p \cdot (-1) \cup W_p \cdot (p-1)$, respectively. The first orbit equals $2p\Z \cup (-2 + 2p\Z)$ and the corresponding block $\P_1(0)$ in $\P_1$ is equivalent to any other regular block (corresponding to another weight in $A$). The second component is $\St_1$, i.e. the semisimple subcategory of $\P_1$ with simple modules $(L_1(-1 +mp))_{m \in \Z}$. 

As we have nothing  more to say about $\St_1$ let us turn to the regular block $\P_1(0)$  associated to $0 \in A$. By Proposition \ref{translated tiltings} we have $T_1(0) = T_{-1}^0 T_1(-1) = T_{-1}^0 \Delta_1(-1) $. Likewise $T_1(2p-2) = T_{p-1}^0 T_1(p-1) = T_{p-1}^0 St_1$. So we see that both $T_1(0)$ and $T_1(2p-2)$ have two $\Delta_1$-factors, namely $\Delta_1(0)$ and $\Delta_1(-2)$, respectively $\Delta_1(2p-2)$ and $\Delta_1(0)$. This allows us to determine all morphisms in $\P_1(0)$ by using the general results from Section \ref{homomorphisms}:

Set $P_m=T_1(0) \otimes 2mp$ and $Q_m = T_1(2p-2) \otimes 2pm$. Write $P = P_0$ and $Q = Q_0$. Then any indecomposable tilting module in $\P_1(0)$ is either isomorphic to $P_m$ or to $Q_m$ for some $m \in \Z$. Using that the dimension of the $\Hom$-space between two tilting modules in $\P_1(0)$ is the number of common $\Delta_1$-factors in their respective $\Delta_1$-filtrations  (all multiplicities are $\leq1$), we get
\begin{equation}
\Hom_{G_1T}(P_m, P_{m'}) \simeq \Hom_{G_1T}(Q_m, Q_{m'}) = \begin{cases} { k^2 \text { if } m = m',} \\ {0 \text { otherwise.} }\end{cases}
\end{equation}
and 
\begin{equation}
\Hom_{G_1T}(P_m, Q_{m'}) \simeq \Hom_{G_1T}(Q_{m'}, P_{m}) =\begin{cases} {k \text { if } |m - m'| = 1}, \\ {0 \text { otherwise.}}\end{cases}
\end{equation}

Let now $u_0$, respectively $u_1$, be a basis element in $\Hom_{G_1T}(P, Q)$, resp. $\Hom_{G_1T}(Q, P_1)$, (in our notation from Section 4 we have $u_0 = \bar g^0(T_1(2p-2)) \in \Hom_{G_1T} (T_1(0), T_1(2p-2))$  and $u_1 = \bar g^{2p-2}(T_1(2p)) \in \Hom_{G_1T}(T_1(2p-2), T_1(2p))$). Then for any $m \in \Z$ we set $u_{2m} = u \otimes 2pm \in \Hom_{G_1T}(P_m, Q_{m})$ and $u_{2m+1} = u_1\otimes 2pm \in \Hom _{G_1T}(Q_m, P_{m+1})$. We let  $d_{2m}$ and $d_{2m+1}$ denote their duals. Note that $u_n$, respectively $d_n$, is the map which takes the top $\nabla_1$-, respectively $\Delta_1$-, factor of its source to the bottom $\nabla_1$-, respectively $\Delta_1$-, factor in its target. So we have the following two exact sequences in $\P_1(0)$:
\begin{equation} \label{longexact}
 \cdots {\overset{u_{-3}} \longrightarrow} P_{-1} {\overset{u_{-2}} \longrightarrow}\ Q_{-1}{\overset{u_{-1}} \longrightarrow} P_0{\overset{u_{0}} \longrightarrow} Q_0 {\overset{u_{1}} \longrightarrow} P_1 {\overset{u_{2}} \longrightarrow} Q_1 {\overset{u_{3}} \longrightarrow} \cdots, 
\end{equation}
and its dual
\begin{equation} \label{duallongexact}
\cdots {\overset{d_{-3}}\longleftarrow} P_{-1} {\overset{d_{-2}}\longleftarrow} Q_{-1}{\overset{d_{-1}}\longleftarrow} P_0 {\overset{d_{0}}\longleftarrow} Q_0 {\overset{d_{1}}\longleftarrow} P_1{\overset{d_{2}}\longleftarrow} Q_1 {\overset{d_{3}}\longleftarrow} \cdots .
\end{equation}

Moreover,  $u_n \circ d_n = d_{n+1} \circ u_{n+1}$ for all $n$ (both these composites map head to socle). Note that $u_n \circ d_n$ is a cellular basis element in its endomorphism ring.

We sum up these findings as follows (using the above notations).

\begin{prop} \label{G_1} The two homomorphisms $u_0$ and $u_1$ generate via formation of direct sums, compositions, taking duals, and tensoring by $\pm 2p$ all morphisms in the regular block $\P_1(0)$ of $\P_1$. The relations satisfied by these generators are:
$$ u_1 u_0 = 0 = (u_0\otimes 2p) u_1 \text { and } u_0  d_0 = d_1  u_1 , \; u_1 d_1 = (d_0  u_0) \otimes 2p.$$
Here $u_0 \otimes 2p = u_2$, $d_0 \otimes 2p = d_2$,  and hence  $ (d_0  u_0) \otimes 2p = d_2  u_2$.
\end{prop}

\begin{rem} 
 The statement in this  proposition about generators for morphisms in $\P_1(0)$  is the easiest case, namely $G= SL_2$ and $r = 1$, of Theorem  \ref{generators for G_r} (restricted to the principal block). The proposition should also be compared to \cite{AT}, Proposition 2.30, which deals with the somewhat similar situation of tilting modules for  $U_q(sl_2)$, $q$ a root of unity in a characteristic zero field.
\end{rem}

\subsection{r $\geq 2$}
First we shall prove that for $SL_2$ we always have the following multiplicity freeness. This result is easy to prove (cf. also the Remark 6.4(1) below) but will nevertheless be important in our approach to finding generators (and  when $r=2$ also relations, cf. Subsection 6.3).
\begin{prop} \label{mult-free}
Let $r \in \Z_{>0}$. Then $(T_r(m):\Delta_r(n)) \in \{0,1\}$ for all $m,n \in \Z$.
\end{prop}

\begin{proof} We shall use induction on $r$. The case $r=1$ is taken care of in the previous subsection. So assume $r > 1$. As the subcategory $\P\St_r$ in $ \P_r$ is equivalent to $\P_{r-1}$ we conclude that the proposition holds for $m \in -1 + p\Z$. So we may assume $m = m_1p + m_0$ with $0\leq m_0 \leq p-2$. Then Proposition \ref{translated tiltings} shows that we may obtain the $\Delta_r$-factors of $T_r(m)$ from the $\Delta_r$-factors of $T_r(m_1p-1)$ as follows. For each $\Delta_r(np-1)$ occurring in $T_r(m_1p-1)$ (this will be with multiplicity $1$ as we have just seen) we get two $\Delta_r$-factors of $T_r(m)$, namely
 those with highest weights in the two alcoves $n p +A$ and $(n-1)p +A$. By the strong linkage principle, see Subsection \ref{SLP}, we get that if $\Delta_r(np-1)$ occurs in $T_r(m_1p-1)$ then $\Delta_r((n\pm1)p - 1)$ do not. Hence $\Delta_r(m')$ cannot occur twice in $T_r(m)$ for any $m' \in \Z$.
\end{proof}

\begin{cor}
Let $r \in \Z_{>0}$ and suppose $m, n \in \Z$. Then $\Hom_{G_rT}(\Delta_r(n), \Delta_r(m))$ has dimension $0$ or $1$.
\end{cor}

\begin{proof} This follows immediately from Proposition \ref{mult-free}: Notice that since $\Delta_r(m) \subset T_r(m)$ we have $\Hom_{G_rT}(\Delta_r(n), \Delta_r(m)) \subset \Hom_{G_rT}(\Delta_r(n), \Delta_r(m))$ and $\dim_k(\Hom_{G_rT}(\Delta_r(n), \Delta_r(m)) = (T_r(m):\Delta_r(n))$.
\end{proof}

\begin{rem}
\begin{enumerate}
\item The composition factor multiplicities $[\Delta_r(m) : L_r(n)]$ are all $0$ or $1$.
This follows  from the proposition by the reciprocity law (\ref{tilt-reciprocity}). Alternatively, it follows from the observation that all weights of $\Delta_r(m)$ have multiplicity $1$. 
\item We could turn the arguments around and use (1) (and its alternative proof) to prove the proposition. The proof we have given has the advantage that it gives a recipe for finding the $\Delta_r$-factors of all $T_r(m)$.
\end{enumerate}
\end{rem}

 As a consequence of Proposition \ref{mult-free} we get that if $(T_r(n):\Delta_r(m))\neq 0$ then there is exactly one cellular basis element in $\Hom_{G_rT}(T_r(m), T_r(n))$ of weight $m$, namely (in our usual notation) $\bar g^m(T_r(n))$. Its dual will be denoted $\bar f^n(T_r(m)) \in \Hom_{G_rT}(T_r(n), T_r(m))
$.  We shall write
$$ u_r(m,n) = \bar g^m(T_r(n)).$$
Recall that by our convention from Remark \ref{choice} we have $u_r(m,m) =  id_{T_r(m)}$. We set
$$B_r = \{u_r(m,n) | 0 \leq m < p^r, \; m \leq n \leq 2p^r-2-m \}.$$
With this notation the result in Theorem \ref{generators for G_r} reads as follows for $G = SL_2$.

\begin{thm}The tuple $(\{P_r(m) \mid 0 \leq m < p^r-1\}, B_r)$ generates $\P_r$ as an additive, $k$-linear category with duality. On the level of morphisms this means that every morphism in $\P_r$ is obtained from $B_r$ by taking direct sums, duals and compositions, and by tensoring with $p^r$. 
\end{thm}

If we restricts ourselves to the principal block $\P_r(0)$ in $\P_r$ associated with $0 \in A$ then this theorem combined with the strong linkage principle give.

\begin{cor} \label{generators for P(0)}The set 
$$B_r(0) = \{u_r(m,n) |\; 0 \leq m < p^r, m \leq n \leq 2p^r-2-m \text { and } n,m \equiv_{(2p)} 0, -2 \}$$
generates the morphisms in $\P_r(0)$.
\end{cor}

\subsection{r = 2}
In this last subsection we shall refine the set of generators coming from the $r=2$ case of Corollary  \ref{generators for P(0)}. Then we go on to find the set of relations satisfied by these generators. Our method will be brute force and it not clear how to generalize to higher $r$.

In analogy with the case $r=1$ the strong linkage principle implies that we have the following components of $\P_2$.
\begin{enumerate}
\item The Steinberg component $\P\St_2$.
\item The  $p$-regular blocks $\P_2(n)$ associated to $n \in A$. 
\end{enumerate}

By (\ref{equiv-St}) the Steinberg component is equivalent to $\P_1$. Therefore case (1) is taken care of by Subsection \ref{r=1}.  

So let us consider (2). As all $p$-regular blocks are equivalent we shall only consider $\P_2(0)$. We set 
$$ P_0 = T_2(0), P_1 = T_2(2p-2), P_2 = T_2(2p), \cdots , P_{2p-1} = T_2(2p^2-2),$$
and if $j = i + 2ap$ with $0 \leq i \leq 2p-1$, $a \in \Z$ 
$$ P_j = P_i \otimes 2ap^2.$$
Using the methods in the proof of Proposition \ref{mult-free} we can now determine the $\Delta_2$-factors of all $P_i$. By the above it is enough to consider $0 \leq i < 2p$:
$$ (P_0:\Delta_2(m)) = 1  \text { if } m \in \{0, -2\} \text { and } (P_p:\Delta_2(m)) = 1  \text { if } m \in \{p^2+p-2, p^2-p\}.$$
If $0<i<p$ then
$$ (P_i:\Delta_2(m)) = 1  \text { if } \begin{cases} { i \text { is even and }  m \in \{ip, ip-2, -ip, -ip-2\}}, \\  { i \text { is odd and }  m \in \{(i+1)p -2, (i-1)p, -(i-1)p-2, -(i+1)p\}}. \end{cases}$$

If $p<i<2p$ then 
$$ (P_i:\Delta_2(m)) = 1  \text { if } \begin{cases}  {i \text { is even and }  m \in \{ip, ip-2, (2p-i)p, (2p-i)p-2\}}, \\  {i \text { is odd and }  m \in \{(i+1)p-2, (i-1)p, (2p-i+1)p-2, (2p-i-1)p\}}. \end{cases}$$
All other $\Delta_2$-multiplicities in $P_i$ are zero.

For each $i \in \Z$ we let $\lambda_i$ denote the highest weight of $P_i$, i.e. 
$$ \lambda_i = \begin{cases} {ip \text { if $i$ is even}}\\ {(i+1)p -2\text { if $i$ is odd.}} 
\end{cases}$$ 

The above formulas for the $\Delta_2$-factors of $P_i$ enable us to find the homomorphisms between any two $P_j$ and $P_{j'}$. Among these we shall now single out the following:
$$ u_i = \bar g_i^{\lambda_i}(P_{i+1}), i =-p, -p+1, \cdots, -2, 0,1, \cdots ,p-2, $$
$$ u'_i = \bar g_i^{\lambda_i}(P_{2p-i}), i = -p+1, -p+2, \cdots , -1, 1, 2, \cdots ,p-1,$$

We denote by $d_i$ and $d'_i$ the corresponding dual elements. Then we have. 

\begin{prop} \label{generators in P_2}
The above homomorphisms 
$$\{u_i | -p \leq i \leq p-2, i \neq -1 \} \cup \{u'_i | -p+1 \leq i \leq p-1, i \neq 0 \}$$ 
generate by taking of direct sums, compositions, duals, and by tensoring with $\pm 2p^2$ all morphisms in $\P_2(0)$. 
\end{prop}

\begin{proof} We shall prove that any cellular basis vector in any $\Hom_{G_2T}(P_j, P_{j'})$ may be expressed - if necessary after tensoring with $\pm 2p^2$ a number of times - as a product of the $u_i, u'_i$ (and their duals $d_i, d'_i$)  listed in the proposition. We may assume that $-p \leq j < p$.

Consider first the case $0 \leq j < p$. Dualizing if necessary we can assume $j \leq j'$.   We now use the general results from Theorem \ref{generators for G_r} combined with the above detailed knowledge of which $\Delta_2$-factors occur in $P_j = T_2(\lambda_j)$. Via the above list of $\Delta_2$-factors of $P_j$ we see that $B_2^{\lambda_j} (\lambda_{j'})$ is empty except for $j' = j, j+1, 2p-j-1, 2p-j, 2p-j+1$ (if $j=0$ only the first three $j'$ occur, and if $j=p-1$ there are only four different such $j'$). We analyze each of these possibilities and claim:
\begin{enumerate}
\item  $B_2 ^{\lambda_j}(\lambda_j) =\{ id_{P_j}, u_{j-1} d_{j-1}, u'_{-j} d'_{-j}, u'_{-j} u_{-j-1} d_{-j-1} d'_{-j}\}$.
\item $B_2^{\lambda_j}(\lambda_{j+1}) = \{u_j, u'_{-j-1} d_{-j-1} d'_{-j}\}.$
\item $B_2^{\lambda_j}(\lambda_{2p-j-1}) = \{ u'_{j+1} u_j\}.$
\item $B_2^{\lambda_j}(\lambda_{2p-j}) = \{u'_j, u'_j u_{j-1} d_{j-1})\}.$
\item $B_2^{\lambda_j}(\lambda_{2p-j+1}) = \{u'_{j-1} d_{j-1}\}.$
\end{enumerate}
Most of the basis vectors listed in these claims come directly from our construction of cellular basis elements. This is the case for the first 3 vectors in (1), the first vector in (2) and (4), and the vector in (5). To see that the remaining vectors are indeed cellular basis vectors we just have to show that 
\begin{equation} \label{non-zero}
 u'_{i+1} u_i \neq 0 \text { for all } i,
\end{equation}
because then these composites as well as their duals are bases for the $1$-dimensional $\Hom$-spaces to which they belong. For instance, the fourth element listed in (1) is (up to a non-zero scalar) equal to $\bar g^{\lambda_{-j-1}}(T_2(\lambda_j)) \bar f^{\lambda_j}(T_2(\lambda_{-j-1})) $.

To check (\ref{non-zero}) we assume $0\leq i < p-1$. Note that $u_i$ is chosen as an extension of $g^{\lambda_i}(T_2(\lambda_{i+1}))$. This homomorphism factors through the inclusion $\Delta_2(\lambda_{i+1}) \hookrightarrow T_2(\lambda_{i+1})$,   because $\Hom_{G_2T}(\Delta_2(\lambda_i), \Delta_2(\lambda_{i+1})) = \Hom_{G_2T}((\Delta_2(\lambda_i), T_2(\lambda_{i+1}))$ (both these $\Hom$-spaces being $1$-dimensional). Likewise, $u'_{i+1}$ is an extension of $g^{\lambda_{i+1}}(T_2(\lambda_{2p-i-1}))$. It will therefore be enough to prove that the composite $g^{\lambda_{i+1}}(T_2(\lambda_{2p-i-1}))\; g^{\lambda_i}(T_2(\lambda_{i+1}))$ is non-zero. As $L_2(\lambda_i)$ is the head of $\Delta_2(\lambda_i)$, it is also the head of the image of $g^{\lambda_i}(T_2(\lambda_{i+1}))$. Now $g^{\lambda_{i+1}}(T_2(\lambda_{2p-i-1}))$ cannot kill this composition factor because the image of $g^{\lambda_{i+1}}(T_2(\lambda_{2p-i-1}))$ must contain the socle of $P_{2p-i-1}$, which is $L_2(\lambda_i)$ (note that $P_{2p-i-1} = T_2(\lambda_{2p-i-1}) \simeq  Q_2(\lambda_i)$ because $\tilde \lambda_i = \lambda_{2p-i-1}$).  

The arguments for $i <0$ are analogous.
\end{proof}

\begin{thm}
The relations satisfied by the generators $u_i, d_i, u'_i, d'_i$ in Proposition \ref{generators in P_2} are the following. 
\begin{enumerate}
\item $u_{i}  u_{i-1} = 0,\;  u_{-i-1} u_{-i-2} = 0  \text { and } d_{i-1} d_{i} = 0, \; d_{-i-2} d_{-i-1} = 0 \text { for }  i=1, \cdots ,p-2.$ 
\item $u'_{i} u'_{-i} = 0, \; d'_{-i-1} d'_{-i} = 0 \text { and } u'_{2p-i} u'_{i} = 0, \; d'_i d'_{2p-i} = 0 \text { for } i=1,2, ,\cdots p-1$.
\item $d_i u_i = u_{i-1}d_{i-1},\;  d_{-i-1} u_{-i-1}  = u_{-i-2} d_{-i-2}   \text { for } i= 1, 2, \cdots , p-2$.
\item $d'_i u'_i  = u'_{-i} d'_{-i}, \; u'_i d'_i = d'_{2p-i} u'_{2p-i} \text { for } i = 1, 2, \cdots , p-1.$
\item Up to non-zero scalars in $k$ we have for $0<i<p-1$
 $$ d_{2p-i-1} u'_i = u'_{i+1} u_i, \; d'_{i} u_{2p-i-1} = d_i d'_{i+1} \text { and } d_i u'_{-i-1} = u'_{-i} u_{-i-1}, \; d'_{-i-1} u_i = d_{-i-1}  d'_{-i}. $$
\item Up to non-zero scalars in $k$ we have for $0<i<p-1$
$$u_{2p-i-1} u'_{i+1} = u'_i d_i , \; d'_{i+1} d_{2p-i-1} = u_i d'_i   \text{ and } u_i u'_{-i} = u'_{-i-1} d_{-i-1}, \; d'_{-i} d_i = u_{-d-1} d'_{-i-1}.$$ 

\end{enumerate}
\end{thm}
Before giving the proof we illustrate the quiver we are dealing with. In the figure below we have chosen $p=7$. Note that horizontally this quiver is infinite periodic: The first two coulums are obtained from the last two by tensoring with the $1$-dimensional $G_2T$-module $-2p^2$ (remember that $p=7$ in the figure). The next two coulumn to the left are obtained by tensoring with one more copy of  $-2p^2$ and so on. In the same way the figure extends to the right by tensoring the last two coulumns with $2p^2$ repeatedly.

\vfill \eject

\[
\begin{tikzcd}
&P_{-14}\arrow{d}{u_{-14}}&&P_0 \arrow[swap]{d}{u_0} \\
\cdots \arrow{r} & P_{-13} \arrow{r}{u'_{-13}} \arrow[swap]{d}{u_{-13}} & P_{-1} \arrow{r}{u'_{-1}} & P_1 \arrow{r}{u'_1} \arrow[swap]{d}{u_1} & P_{13} \arrow{r}  & \cdots \\
\cdots \arrow{r}  & P_{-12} \arrow{r}{u'_{-12}} \arrow[swap]{d}{u_{-12}} &P_{-2} \arrow{u}{u_{-2}} \arrow{r}{u'_{-2}} & P_{2} \arrow{r}{u'_{2}} \arrow[swap]{d}{u_2} & P_{12} \arrow[swap]{u}{u_{12}} \arrow{r}&{\cdots}\\
\cdots \arrow{r}  & P_{-11} \arrow{r}{u'_{-11}} \arrow[swap]{d}{u_{-11}}  &P_{-3} \arrow{u}{u_{-3}} \arrow{r}{u'_{-3}} & P_{3} \arrow{r}{u'_{3}} \arrow[swap]{d}{u_3}&  P_{11} \arrow[swap]{u}{u_{11}} \arrow{r}&{\cdots}\\
\cdots \arrow{r} & P_{-10} \arrow{r}{u'_{-10}} \arrow[swap]{d}{u_{-10}}  & P_{-4} \arrow{u}{u_{-4}} \arrow{r}{u'_{-4}} & P_{4} \arrow{r}{u'_{4}}\arrow[swap]{d}{u_4} & P_{10} \arrow[swap]{u}{u_{10}} \arrow{r} & {\cdots}\\
\cdots \arrow{r} & P_{-9} \arrow{r}{u'_{-9}} \arrow[swap]{d}{u_{-9}} & P_{-5} \arrow{u}{u_{-5}} \arrow{r}{u'_{-5}} & P_{5} \arrow{r}{u'_{5}} \arrow[swap]{d}{u_5}& P_{9} \arrow[swap]{u}{u_{9}} \arrow{r} & {\cdots}\\ 
\cdots \arrow{r}  & P_{-8} \arrow{r}{u'_{-8}}  & P_{-6} \arrow{u}{u_{-6}} \arrow{r}{u'_{-6}} & P_{6} \arrow{r}{u'_{6}}  &P_{8} \arrow[swap]{u}{u_{8}} \arrow{r} & {\cdots}\\
&& P_{-7} \arrow{u}{u_{-7}} &{} &P_7 \arrow[swap]{u}{u_{7}} &{}
\end{tikzcd}
\]

$$ \text  {\it Quiver for $\P_2(0)$ with $p=7$}$$

\vskip .5 cm
Let us call $u_i$ and $u'_i$ uparrows while $d_i$ and $d'_i$ are called downarrows. Then in terms of the above diagram relations (1) and (2) say that the composite of two consecutive horizontal uparrows are $0$ and that the same is true vertically as well as when we replace up by down. Relations (3) means that horizontal loops at a vertex (paths of length two with the vertex as starting and ending point) are identical, and (4) expresses the same for vertical loops. The first relations in (5) say that the diagrams 
\[
\begin{tikzcd}
P_i \arrow{r}{u'_i} \arrow[swap]{d}{u_i} & P_{2p-i} \arrow{d}{d_{2p-i-1}}\\
P_{i+1} \arrow{r}{u'_{i+1}} & P_{2p-i-1} 
\end{tikzcd}
\]
commute. The other relations in (5) and (6) are equivalent to the commutativity of similar diagrams.

\begin{proof}
The relations (1) and (2) follow from the fact that the corresponding $\Hom$-spaces are zero, cf. Proposition \ref{hom-bounds} and the above determinations of $\Delta_2$-factors of the $P_i$'s.  

The relations in (3) come from the corresponding relations among the $G_1T$-homomorphisms discussed in Subsection 6.1: 
We have $T_2(mp) \simeq T(p)\otimes T_1(m-1)^{(1)}$ and $T_2(mp-2) \simeq T(2p-2) \otimes T_1(m-2)^{(1)}$ for all $m \in \Z$. Here $T(p)$ and $T(2p-2)$ are the indecomposable tilting modules for $G$ with highest weights $p$ and $2p-2$. As $G_1T$-modules they identify with $T_1(p) \simeq T_1(0) \otimes p$ and $T_1(2p-2)$, respectively. As $G_1$ acts trivially on twisted modules we therefore have for all $m$ not divisible by $p$ the following $G_1$-isomorphisms.
$$ P_m|_{G_1} \simeq \begin{cases} { (T(p) \otimes T_1(m-1)^{(1)})|_{G_1} \simeq T_1(0)|_{G_1} \otimes k^{2p} \text { if $m$ is even, }} \\ {(T(2p-2) \otimes (T(m-1)^{(1)})|_{G_1} \simeq T_1(2p-2)|_{G_1} \otimes k^{2p} \text { if $m$ is odd.}} \end{cases} $$ 
Consider the even case. It follows from the above identifications that  the restriction to $G_1$ of $u_m:P_m \rightarrow P_{m+1}$ is $2p$ copies of the map $u: T_1(0) \rightarrow T_1(2p-2)$ considered in Subsection 6.1. Analogously its dual $d_m$ restricts to $2p$ copies of $d$, the dual of $u$.  Similar arguments applies to the restrictions of $u_{m-1}$ and $d_{m-1}$. As the odd case is completely analogous we see that  (3) is a consequence of the relation $d u = u_{-1} d_{-1}$, which holds because of Proposition \ref{G_1}.

To prove the relations in (4) we first recall from (5.3) that the indecomposable tilting modules in $\St_2$ have the form $T_2(p-1 + pm) \simeq St_1 \otimes T(m)^{(1)}$, $m \in \Z$. If we only consider those with $m \in W_p \cdot 0$ we see from (\ref{longexact}) that they fit into a long exact sequence
\begin{equation}\label{St-longexact}
\cdots \rightarrow T_2(p-1-2p^2) \rightarrow T_2(p-1-2p) \rightarrow T_2(p-1) \rightarrow T_2(p-1+p(2p-2)) \rightarrow \cdots.
\end{equation}
When we apply the translation functor onto the block $B_2(0)$ to $T_2(p-1+pm)$ we obtain $T_2(2p-2+pm)$ if $m$ is even, and $T_2(p(m+1))$ when $p$ is odd. In the notation used in this section this means that the sequence (\ref{St-longexact}) gives rise to the long exact sequence
$$ \cdots \rightarrow P_{-2p+1} \rightarrow P_{-1} \rightarrow P_1 \rightarrow P_{2p-1}  \rightarrow \cdots .$$
The homomorphisms occurring in this sequence are $\cdots u'_{-2p-1}, u'_{-2p+1}, u'_{-1}, u'_1, u'_{2p-1}, \cdots $ and in the dual sequence they are the corresponding  $d'_j$. When $i =1$ the relations in (4) are therefore a consequence of the corresponding relations in $\P_1(0)$, see the relation in the paragraph following (\ref{longexact}) and (\ref{duallongexact}).

If in this argument we replace $W_p\cdot 0$ by $W_p \cdot m$ we obtain in the exact same way the relations in (4) for $i = m+1 = 2, 3, \cdots , p-1$.

Let us now prove the first relation in (5): In (\ref{non-zero}) we saw that $u'_{i+1} u_i \neq 0$. We shall see that also $d_{2p-i-1} u'_i \neq 0$. This will prove the relation because $\Hom_{G_2T}(P_i, P_{2p-i-1}) \simeq k$. Our proof of (4) shows that the image of $u'_i$ is the submodule  $E \subset P_{2p-i}$ with $\nabla_2$-factors $\nabla_2(\lambda_{i-1})$ and   $\nabla_2(\lambda_{i})$. So if $d_{2p-i-1}$ kills this image then it must belong to $\Hom_{G_2T}(P_{2p-i}/E, P_{2p-i-1})$. However, the socle $L(\lambda_i)$ of $P_{2p-i-1}$ occurs only once as a composition factor of $P_{2p-i}/E$, so this means that the image of $d_{2p-i-1}$ consists only of the two composition factors $L_2(\lambda_i)$ and $L_2(\lambda_{i-1})$ (the latter being the head of $P_{2p-i}$). But this is inconsistent with the construction of $d_{2p-i-1}$ which says that its composition with the surjection $P_{2p-i-1} \rightarrow \nabla_2(\lambda_{2p-i-1})$ is non-zero. This forces $L_2(\lambda_{2p-i-1})$ to be a composition factor of the image of $d_{2p-i-1}$. 

The second relation in (5) is the dual of the first. The third relation is proved in the same way as the first, and finally the fourth relation is dual to the third.

Finally we prove the first relation in (6) (then the remaining relations in (6) are - just like for (5) - either proved analogously or by duality): We will argue similarly as for the first relation in (5). The image of $u'_{i+1}$ is the submodule $E \subset P_{2p-i-1}$ consisting of the two $\nabla_2$-factors $\nabla_2(\lambda_i)$ and $\nabla_2(\lambda_{i+1})$. If $u_{2p-i-1}$ kills this image then it must belong to  $\Hom_{G_2T}(P_{2p-i-1}/E, P_{2p-i})$. Since the socle of $P_{2p-2}$, namely $L_2(\lambda_{i-1})$, occurs only once in $P_{2p-i-1}/E$ the image of any non-zero map in $\Hom_{G_2T}(P_{2p-i-1}/E, P_{2p-i})$ must have composition factors just $L_2(\lambda_{i-1})$ and  $L_2(\lambda_i)$. This contradicts the fact that $L_2(\lambda_{2p-i-1})$ must occur in the image of $u_{2p-i-1}$, because   $u_{2p-i-1}$ is non-zero on $\Delta_2(\lambda_{2p-i-1})$.

Having established the relations (1) - (6) we have left to show that they suffice, i.e. that any path $\underline p$ in the uparrows and downarrows can be rewritten by these relations in terms of the known basis vectors for the morphisms in $\P_2(0)$. First we observe that relations (3), (4) and (5) allow us to move all downarrows in $\underline p$ to the right of all the uparrows. Suppose now that there are no downarrows in $\underline p$. Then $\underline p$ has length at most $3$. In fact, by (1) and (2) the path will be zero unless it changes direction at all vertices, i.e. it has to alternate between $u$'s and $u'$'s.  Suppose it starts at vertex $i$. Then its first arrow is either $u_i$ of $u'_i$. Assume that $0 \leq i\leq p-1$ We claim that this leads to the following possibilities
\begin{enumerate}
\item $ u_0, u'_1 u_0 \text { if } i = 0,$
\item $u_i, u'_i, u'_{i+1}u_i, u_{2p-i} u'_i,  u_{2p-i-1}u'_{i+1} u_i \text { if } 0 < i < p-1,$ 
\item $u'_{p-1}, u_{p+1} u'_{p-1} \text { if } i = p-1.$
\end{enumerate}
Consider the ``generic" case $0 < i < p-1$. Note that the first $3$ paths in the list here are cellular basis elements. So is the $4$'th according to the first relation in (6). The $5$'th path can be rewritten as follows via the first relations in (6) and (3): $u_{2p-i-1}u'_{i+1} u_i  = u'_i d_i u_i = u'_i u_{i-1} d_{i-1}$. The last expression is a cellular basis element. On the other hand, we cannot extend $  u_{2p-i} u'_i$ non-trivially: by the third relation in (5) (shifted by $2p$) we get  $u'_{2p-i+1} u_{2p-i} u'_i = d_{2p+i} u'_{2p+i} u'_i = 0$. A similar argument shows that the $5$'th path cannot be extended. So the $5$ listed paths exhaust all possibilities. In the two ``non-generic" cases, $i = 0$ and $ i = p-1$ the two given paths clearly are cellular basis elements. The last one cannot be extended, so there are no other paths.

Dualizing these uparrow paths we get the corresponding list of downarrow paths ending at vertex $i$. In the case $0 < i < p-1$ the list contains the following $5$ paths: $ d_i, d'_i, d_i d'_{i+1}, d'_i d_{2p-i}, d_i d'_{i+1} d_{-2p-i-1}.$ We have to show that if we take any  of these $ 5$ downarrow paths and compose with any of the $5$ uparrow paths we get via relations (1)-(6) either $0$ or a path that represents a (linear combination of) cellular basis element(s). We leave this task to the reader (as well as the corresponding ``non-generic" cases) giving only the following $2$ examples: 

Consider $u'_{i+1} u_i d_i d_{i+1}$. This is a cellular basis element in $\End_{G_2T}(P_{2p-i-1})$. In fact, by (\ref{non-zero}) we know that $u'_{i+1} u_i \neq 0$. Hence it equals (up to a non-zero scalar) $\bar g^{\lambda_i}(T_2(\lambda_{2p-i-1}))$. The dual element $d_i d'_{i+1}$ is therefore proportional to $\bar f^{\lambda_{2p-i-1}}(T_2(\lambda_{i}))$ and hence $u'_{i+1} u_i d_i d_{i+1}$ represents the cellular basis element $\bar g^{\lambda_i}(T_2(\lambda_{2p-i-1})) \bar f^{\lambda_{2p-i-1}}(T_2(\lambda_{i}))$.

As the second example we choose $ u_{2p-i} u'_i d_i d'_{i+1} d_{-2p-i-1}$. The first relation in (6) allow us to rewrite our path as $u'_{i-1} d_{i-1}  d_i d'_{i+1} d_{-2p-i-1}$. But this is $0$ because $d_{i-1} d_i = 0$ according to relation (1). 

In the process of checking the remaining $23$ cases the reader may observe that all paths of length more than $4$ are $0$. Moreover, a non-zero paths of length $4$ is a  loop and may be rewritten as a path around any of the four squares in the diagram containing the vertex of the loop.

\end{proof}

\vskip 1 cm 

\begin{thebibliography}{}

\bibitem{SLP} H.H.~Andersen, The strong linkage principle, J. Reine Ang. Math. 315 (1980), 53-59.
\bibitem{A98}H.H. Andersen, Tilting modules for algebraic groups, in: Algebraic Groups and their Representations (editors: R.W. Carter and J. Saxl), 25-42, Nato ASI Series, Series C 517, Kluwer (1998).
\bibitem{A00} H.H.~Andersen, A sum formula for algebraic groups, (Proc. of conference in honor of D. Buchsbaum, Roma 1998), J. Pure and Appl. Alg. 152 (2000), 17-40.
\bibitem{A18} H.H.~Andersen, The Steinberg linkage class for a reductive algebraic group, Arkiv f\"{o}r Matematik 56 (2018), 229-241.
\bibitem{AJS} H.H.~Andersen, J.C.~Jantzen and W.~Soergel, Representations of quantum groups at a pth root of unity and of semisimple groups in characteristic p: independence of p, Ast\'{e}risque 220 (1994), 321 pages.
\bibitem{AJL} H.H.~Andersen, J. J\o rgensen and P. Landrock, The projective indecomposable modules for $SL(2,p^n)$, Proc. London Math. Soc. XLVI (1983), 38-52.
\bibitem{AT} H.H. Andersen and D. Tubbenhauer, Diagram categories for Uq-tilting modules at roots of unity, Transform. Groups 22 (2017), 29-89.
\bibitem{AST} H. H. ~Andersen, C. Stroppel and D.~Tubbenhauer, Cellular structures using Uq-tilting modules, Pacific J. Math. 292 (2018), 21-59.
\bibitem{BNPS} C.P. Bendel, D.K. Nakano, C. Pillen, and P. Sobaje, Counterexamples to the Tilting and (p, r)-Filtration Conjectures,  J. Reine Angew. Math. (to appear).
\bibitem{BS} J. Brundan and C. Stroppel, Semi-infinite highest weight categories (2018), online available on arXiv:1808.08022.
\bibitem{Cu} C. W. Curtis, Representations of Lie algebras of classical type with applications to linear algebraic groups, J. Math. Mech. 9 (1960), 307-326.
\bibitem{Do} S. Donkin, On tilting modules for algebraic groups, Math. Z. 212 (1993), 39-60.
\bibitem{EL} B. Elias and A. Lauda, Trace decategorification of the Hecke category, Journal of Algebra 449 (2016),  615-634.
\bibitem{GL} J. Graham and G. Lehrer, Cellular algebras, Inventiones Math. 123 (1996), 1-34.
\bibitem{Hu} J.E. Humphreys, Representations of Semisimple Lie Algebras in the BGG Category $\O$, Graduate Studies in Mathematics 94, American Mathematical Society (2008).
\bibitem{HV} J.E. Humphreys and D.-n. Verma, Projective modules for finite Chevalley groups, Bull. Amer. Math. Soc. 79 (1973), 467-468.
\bibitem{RAG} J.C.~Jantzen, {\em Representations of Algebraic Groups}, Mathematical Surveys and Monographs 107, Second edition, American Mathematical Society (2003).
\bibitem{FM} F. Marko, Algebra associated with the principal block of category $\O$ for sl3(C), in: Algebras, rings and their representations, 201-214, World Sci. Publ., Hackensack, NJ  (2006).

\bibitem{Ma} O. Mathieu, Filtrations of G-modules, Ann. scient. Ec. Norm. Sup. 23 (1990), 625-644.
\bibitem{RW} S.~Riche and G.~Williamson, A simple character formula (2019), online available on arXiv:1904.08085.
\bibitem{S} P. Sobaje, The power of tilting module characters (2019), online available on arXiv:1902.10308.
\bibitem{St} R. Steinberg, Representations of algebraic groups, Nagoya Math. J. 22 (1963), 33-56.
\bibitem{Str} C. Stroppel, Category $\O$: Quivers and endomorphism rings of projectives, Representation Theory, Elec. J. of Amer. Math. Soc. 7 (2003), 322-345. 
\bibitem{TW} D.~Tubbenhauer and P. Wedrich, Quivers for SL(2) tilting modules (2019), online available on arXiv:1907.11560.
\bibitem{W} B. R.~Westbury, Invariant tensors and cellular categories, Journal of Algebra 321 (2009),
 3563-3567.

 \end{thebibliography}
\end{document}